%% file: mixed.tex
\documentclass{amsart}
\usepackage{amsmath, amsthm, amssymb}
\usepackage{graphicx}
\usepackage[dvipsnames]{xcolor}
\usepackage{pgfplots}
\usepackage{ifthen}

\usepackage[utf8]{inputenc}
\usepackage[author-year]{amsrefs}
\usepackage{enumerate}
\newtheorem{theorem}{Theorem}

\newtheorem{proposition}[theorem]{Proposition}
\newtheorem{lemma}[theorem]{Lemma}

\theoremstyle{definition}
\newtheorem{definition}[theorem]{Definition}
\theoremstyle{remark}
\newtheorem{remark}[theorem]{Remark}
\newtheorem{claim}[theorem]{Claim}
\newtheorem*{remark*}{Remark}
\newtheorem*{note*}{Note}
\newtheorem{example}[theorem]{Example}

\usepackage{calc}

\newcommand{\rank}[1]{\mathrm{rank}\left(#1\right)}
\newcommand{\conv}[1]{\mathrm{Conv}(#1)}

\newcommand{\avg}[1]{\mathrm{Avg}_{#1}}

\newcommand{\argmin}{\mathrm{argmin}}

\newcommand{\Aact}{\mathbf C_{\mathrm{act}}}
\newcommand{\Ainact}{\mathbf C_{\mathrm{inact}}}
\newcommand{\bact}{\mathbf b_{\mathrm{act}}}
\newcommand{\binact}{\mathbf b_{\mathrm{inact}}}
\newcommand{\Vknown}{\mathcal V_{\mathrm{known}}}
\newcommand{\Vexplore}{\mathcal V_{\mathrm{pending}}}

\newcommand{\defeq}{\stackrel{\scriptsize\mathrm{def}}{=}}

\newcommand{\bigmatrix}{\ensuremath{\left[\begin{matrix}
\begin{matrix}
-1 & 0 & \dots & 0\\
\vdots & \vdots && \vdots \\
-1 & 0 & \dots & 0\\
\end{matrix}       & \hspace{3em} & \Huge{A_1}& \hspace{3em}\\
\begin{matrix}
0 & -1 & \dots & 0\\
\vdots & \vdots && \vdots \\
0 & -1 & \dots & 0\\
\end{matrix}       & \hspace{3em}& \Huge{A_2}& \hspace{3em}\\
\ddots                              & &\vdots& \\
\begin{matrix}
0 & 0 & \dots & -1\\
\vdots & \vdots && \vdots \\
0 & 0 & \dots & -1\\
\end{matrix}       & \hspace{3em}& \Huge{A_s}& \hspace{3em}\\
\end{matrix}
\right]
}}

\newcommand{\rev}[1]{{#1}}
\newcommand{\reva}[1]{{#1}}
\newcommand{\revb}[1]{{#1}}
\newcommand{\revc}[1]{{#1}}
\newcommand{\revd}[1]{{#1}}

\newcommand{\vol}[1]{\reva{\mathrm{Vol}(}{#1}\reva{)}}
\newcommand{\mm}[2]{\reva{m_{#1}^{(#2)}}}

\makeglossary
\newenvironment{theglossary}{\begin{tabular}{lll}}{\end{tabular}}
\newcommand{\glossaryentry}[2]{#1 &\ifthenelse{\equal{#2}{}}{\\}{p.{#2}\\}}
\renewcommand{\MR}[1]{}

\author{Gregorio Malajovich}
\thanks{A substantial part of this paper was written while visiting the Simons Institute for the Theory of Computing in the
University of California at Berkeley. This visit was funded by CAPES (Coordenação de 
Aperfeiçoamento de Pessoal de Nível Superior, Brazil. Proc. BEX 2388/14-6). This research is  
also funded by CNPq, grants 441678/2014-9 and 306673/2013-4. Numerical experiments were performed at 
NACAD (Núcleo Avançado de Computação de Alto Desempenho) at UFRJ}
\address{Departamento de Matemática Aplicada, Instituto de Matemática, Universidade Federal do
Rio de Janeiro. Caixa Postal 68530, Rio de Janeiro RJ 21941-909, Brasil.}\email{gregorio.malajovich@gmail.com}
\title[Computing mixed volume...]{Computing mixed volume and all mixed cells in quermassintegral time}
\date{April 14, 2016}

\begin{document}
\begin{abstract}
The mixed volume counts the roots of generic sparse polynomial systems. Mixed cells are used to
provide starting systems for homotopy algorithms that can find all those roots, and track no
unnecessary path.
Up to now, algorithms for that task were of enumerative type, with no general non-exponential complexity bound.
A geometric algorithm is introduced in this paper. Its complexity is bounded in the average and probability-one
settings in terms of some geometric invariants: quermassintegrals associated to the tuple of convex hulls of 
the support of each polynomial.
Besides the complexity bounds, numerical results are reported. Those are consistent with an output-sensitive running time for
each benchmark family where data is available. For some of those families,
an asymptotic running time gain over the 
best code available at this time was noticed.
\end{abstract}
\maketitle

\tableofcontents
\section{Introduction}

The {\em mixed volume} of \revb{an} $n$-tuple of convex bodies $(\mathcal A_1, \dots, \mathcal A_n)$,
$\mathcal A_i \subset \mathbb R^n$ is defined by
\[
\glossary{$V(\mathcal A_1, \dots, \mathcal A_n)$&Mixed volume of $\mathcal A_1, \dots, \mathcal A_n$.}
V({\mathcal A_1}, \dots, {\mathcal A_n})
\defeq
\frac{1}{n!}
\frac{\partial^n}
{\partial t_1 \partial t_2 \cdots \partial t_n}
\vol{t_1 \mathcal A_1 + \cdots + t_n \mathcal A_n} 
\]
where $t_1, \dots, t_n \ge 0$ and the derivative is computed at $t=0$. It generalizes 
ordinary volume:
\[
V(\mathcal A, \mathcal A, \dots, \mathcal A) = \vol{A}.
\] 
\par
Mixed volume was introduced by \ocite{Minkowski} in connection with 
the {\em quermassintegrals} $V(\mathcal A, \mathcal A, B^3)$ and $V(\mathcal A,
B^3, B^3)$\rev{, where $B^3$ stands for the \reva{unit} 3-ball. Those quermassintegrals} are equal (up to a factor) \reva{to} the area and the total mean curvature of $\partial \mathcal A$.
\medskip
\par
In this paper, $A_1, \dots, A_n$ are finite subsets of $\mathbb Z^n$.
We will provide an algorithm to compute 
the scaled mixed volume 
\[
\glossary{$A_i$&Finite subset of $\mathbb Z^n$.}
V= n!\ V( \conv{A_1}, \cdots, \conv{A_n})
\glossary{$V$ & Scaled mixed volume $n!V( \conv{A_1}, \cdots, \conv{A_n})$.}
\]
together with a set of lower mixed facets for a random lifting (in modern language, a 
zero-dimensional tropical variety). The BKK bound \cites{Bernstein, BKK}
states that $V$ is the number
of roots in $(\mathbb C^{\times})^n$ of a generic system of Laurent polynomials 
$f_1(\mathbf x)=\cdots = f_n(\mathbf x) = 0$ where
\begin{equation}\label{sparse-system}
f_{i}(\mathbf x) = \sum_{\mathbf a \in A_i} f_{ia}x_1^{a_1}x_2^{a_2}\cdots x_n^{a_n} 
\hspace{1em}, 1 \le i \le n
\end{equation}
\revc{where the $f_{ia}$ are complex numbers.}
\ocite{Huber-Sturmfels} suggested to use 
the lower mixed facets (Def.\ref{def:mixed}) 
to produce a starting system for homotopy
algorithms to solve sparse polynomial equations $f_1(\mathbf x)=\cdots = f_n(\mathbf x)=0$ with $f_i$ as above.
\ocite{Emiris-Canny} introduced the first incremental algorithms for computing mixed volume 
and mixed cells.
For a certain time, computing the starting system was a bottleneck for 
\rev{polyhedral} 
homotopy based polynomial solving software \rev{\cite{Lee-Li}*{p.98}}. 
Later breakthroughs by
~\ocites{Gao-Li, Li-Li, Gao-Li-Mengnien, Mizutani-Takeda-Kojima, Lee-Li,
Chen-Lee-Li}
provided efficient practical implementations through enumerative 
algorithms (Remark~\ref{predecessors}).
The complexity properties of those   
algorithms are not well understood. 
\par
The algorithm {\sc AllMixedCells} in page \pageref{allmixedcells}
is geometric in nature. This will allow for a complexity
bound in terms of geometric invariants (quermassintegrals).
\medskip
\par
Before writing an algorithm or stating complexity results, 
one should fix a model of computation. In this paper, an {\em algorithm} is
a randomized \reva{{\em real Random Access Machine} (real RAM) \cite{Preparata-Shamos}}. 
Arithmetic operations
$+$, $-$, $\times$, $/$,$\sqrt{\ },$\reva{$\log()$, $\sin()$, $\cos()$} 
are allowed and cost one unit of time.
Memory access is also assumed to be performed at unit cost. 
\reva{In addition, a {\em randomized} real RAM has 
access to an unlimited supply of independently uniformly distributed
random numbers in $[0,1]$.}
\rev{The running time of a machine with 
a fixed input is therefore a random variable. Henceforth, the expressions 
{\em with probability one} and
\reva{{\em on average}} refer to the product measure of $[0,1]^{\infty}$}.

\medskip
\par

Let $\mathcal A_i = \conv{A_i}$, $\mathcal A = \mathcal A_1 + \cdots 
+ \mathcal A_n$ and
let $B^n$ be the unit $n$-ball of radius 1.
Let $d_i=\dim \conv{A_i}$. Let $V_i = d_i !\vol{\conv{A_i}}_{d_i}$ be the
generic root bound of an unmixed polynomial system of support $A_i$.
Let $0 \le E_i < \# A_i$ be \reva{the} numbers to be \reva{formally} defined in section~\ref{sec:algorithm} \reva{(but see Remark~\ref{rem:Ei} below)}.
\begin{theorem}\label{main}
\glossary{$V_i$ & Generic root bound of an unmixed system of support $A_i$.}
With probability one, 
the algorithm {\sc AllMixedCellsFull} \revb{stated on page \pageref{allmixedcellsfull}}
produces all the lower mixed facets in time bounded by $O(T + T')$ arithmetic operations, where

\begin{equation}\label{main-time1}
T = (\sum_{i=2}^{\reva{n}} v_i) \left( n^2 \sum_{i=1}^n E_i + \log \sum_{i=2}^n v_i \right),
\glossary{$T$, $T'$&Time bounds for the algorithm}
\end{equation}

\begin{equation}\label{main-time2}
T' = (\max V_i) \left( n^2 \sum_{i=1}^n \# A_i + \log \max_{i=1, \dots, n} V_i \right),
\end{equation}
and \reva{$v_i$ is a random variable satisfying the two bounds below:}
\begin{enumerate}[(a)]
\item With probability one,
\[
v_{\reva{i}} \le
n! V(\mathcal A_1, \dots, \mathcal A_{i-1}, \mathcal A, B^n, \dots, B^n).
\]
\item \reva{Let $\bar v_i$ be the average of $v_i$, then}
\[
\bar v_i 
\le \frac{n!}{2^{n-i}} V(\mathcal A_1, \dots, \mathcal A_{i-1}, \mathcal A, B^n, \dots, B^n) 
.
\]
\end{enumerate}
\end{theorem}

\begin{remark}
If the polytopes $A_i$ are 
represented by dense $\#A_i \times n$ matrices, then $n \sum E_i$
is a lower bound for the input size $S$. So the complexity can be
bounded above by
\begin{equation}\label{main-time}
O \left(W ( n S + \log W) \right),
\hspace{2em}
W=\max\left( \max_i( V_i) \ ,\  \sum_i v_i \right)
.
\end{equation}
\end{remark}

\begin{remark}
Because of monotonicity of the mixed volume,
$v_{\reva{i}} \le n!V(\mathcal A, \dots, \mathcal A,$ $B^n, \dots, B^n)$. 
Also, $V_i \le n! \vol{\mathcal A}_n$. If
$\mathcal A$ contains a copy of the unit ball, \revb{then} $v_{\reva{i}} \le n!\vol{\mathcal A}$.
\end{remark}

\begin{remark}
In the probability-one bound (a) for $v_{\reva{i}}$, one can replace $B^n$ by
$\beta_n=\{x \in \mathbb R^n: \|x\|_1 \le 1\}$, the $n$-orthoplex (Sec. \ref{sec:deterministic}). \revc{This replacement gives an exponentially smaller bound
when $i$ is small.}
Assuming that $\vol{\mathcal A} \ne 0$, \revc{we obtain}
\[
v_{\reva i} \le n! 2^{n-\reva{i}} \vol{\mathcal A}
.
\]
It is not clear whether a similar bound holds for \reva{$\bar v_i$ in }the average case analysis (b).
\end{remark}

\begin{remark}
Assume that $\dim \conv{A_i} = n$. Let $\delta_i$ denote the radius of the inscribed sphere to $\mathcal A_i$ and
$\Delta$ the radius of the circumscribed sphere to $\mathcal A$.
Then,
\[
v_{\reva{i}} \le n! \frac{\Delta}{\delta_{\reva{i}} \delta_{\reva{i}+1} \cdots \delta_n} V( \mathcal A_1, \dots, \mathcal A_n)
\hspace{1em}
\text{ and }
\hspace{1em}
\bar v_{\reva i} \le \frac{n!}{2^{n-\reva{i}}} \frac{\Delta}{\delta_{\reva{i}} \delta_{\reva{i}+1} \cdots \delta_n} V( \mathcal A_1, \dots, \mathcal A_n)
.
\]
\end{remark}

\begin{remark}\label{rem:Ei}
\reva{The bound }$T'$ \reva{represents} the cost of computing a lower convex hull of a random lifting for
each of the polytopes $A_1,
\dots, A_n$.
Typically $T' \ll T$ but counterexamples may be produced. 
In a previous version of this paper, the algorithm
was assumed to receive those lower convex hulls as precomputed information.
$E_i$ is the degree of the 1-skeleton of the lower convex hull for the
lifting of $A_i$ (Sec.~\ref{sec:algorithm}).
\end{remark}

From a complexity standpoint, bounding the cost of mixed volume 
computation in terms of the mixed volume and similar invariants
is the best that we can aim for.
The general problem
of computing the mixed volume is known to be \#P-complete. This follows
from the famous result by \ocite{Khachiyan} that computing
{\em volumes} of convex polytopes is already \#P-hard. 
\ocite{Barvinok} suggested approximating mixed volumes by the
mixed volume of ellipsoids. \ocite{Gurvits} obtained an approximation
within a factor exponential in the number of variables, and showed
that the same ratio could \revd{not} be obtained with a deterministic 
algorithm in the Oracle setting. \ocite{Dyer-Gritzmann-Hufnagel}
provided good approximations in certain special cases, but
showed also that computing mixed volumes of zonotopes is already
\#P-hard.

\ocite{Emiris} was able to bound the complexity of the
algorithm by \ocite{Emiris-Canny} for enumerating mixed cells
in terms of the volume of the Minkowski sum of all polytopes. 
Assuming that all polytopes have non-zero $n$-dimensional volume,
he deduced bounds for the bit-complexity. \rev{Simultaneously,
\ocite{Verschelde-Gatermann-Cools} introduced dynamic lifting and also
obtained complexity bounds. In both papers the complexity bounds depend
on the number of lower facets, not necessarily mixed facets. More recently}
\ocite{Emiris-Vidunas} gave specific formulas for certain
semi-mixed volumes. \ocite{Emiris-Fisikopoulos} 
\rev{devised} an algorithm to compute the mixed
volume without \rev{actually computing} the mixed cells.
\bigskip
\par
\rev{The algorithm in this paper visits $v_{n}$ lower facets including
all the mixed cells. Those facets are `dual' to a certain tropical curve,
not necessarily connected. The precise definition of this tropical curve
requires the introduction of a mixed Legendre transform, which allows to
efficiently represent tropical varieties \revb{as} 
specific subsets of the viable set 
of some linear programming problem. The precise formalism in introduced in
Section~2.}
\par
\rev{In Section~3, it is proved that each connected component of this
tropical curve cuts a generic affine hyperplane with probability 1.
This allows to find all the connected components by a dimensional induction.}
\rev{Each component of the tropical curve is explored by a particular
pivoting procedure, that takes into account the structure of the problem.
Those procedures are explained in Section 4, together with the procedures
for pivoting from one induction level to the other. Because of numerical
stability reasons, the generic affine hyperplane is sent to infinity and 
the pivoting procedures use nonstandard real numbers (real polynomials \reva{in}
 a parameter $R \rightarrow \infty$).}
\par
\rev{The algorithm and intermediate complexity bounds are given in
Section 5. An important complexity gain is obtained by assigning
a {\em hash value} to every lower \revc{face}. This allows to efficiently
store sets of explored and unexplored lower \revc{faces} as a balanced tree.
}
\rev{
The proof of Theorem~\ref{main} is completed in Sections 
\ref{sec:deterministic} and \ref{sec:average}, where the numbers
$v_d$ and $\bar v_d$ get bounded in terms of mixed volumes
(quermassintegrals).
}
\par
An actual implementation of the algorithm is described in
sections~\ref{sec:implementation} to \ref{sec:conclusions}. 
This is part of a long-term project to produce a toric homotopy
based polynomial system solver. The source code is available 
at \url{http://sourceforge.net/projects/pss5/} and licensed 
under GNU Public License. At this time, only the mixed volume section of the
code is complete and fully tested. \rev{There is a glossary of notations
at the end of the paper}.

The rationale for including sections \ref{sec:implementation} to
~\ref{sec:conclusions} in this paper is to substantiate the following
claims:

\begin{claim} The model of computation is realistic, in the sense that
the complexity bound in Theorem~\ref{main} accurately describes the
running-time measurements for a publicly available implementation of the
algorithm. 
\end{claim}

\begin{figure}
\begin{tikzpicture}
\begin{loglogaxis}[
	xlabel={Value of $T$},
	ylabel={Measured time (s)},
        grid=major,
        legend style={cells={anchor=west}, rounded corners=2pt, at={(1.5,1.0)}}
        ]
\addplot +[only marks] coordinates{(7.51E+10,2.06E+02)(3.66E+11,8.50E+02)(1.96E+12,4.07E+03)};
\addplot +[only marks] coordinates{(8.92E+11,1.23E+03)(2.24E+12,2.87E+03)(5.62E+12,6.46E+03)};
\addplot +[only marks] coordinates{(3.45E+11,5.18E+02)(9.27E+11,1.27E+03)(2.50E+12,3.08E+03)(6.67E+12,7.58E+03)};                                 
\addplot +[only marks] coordinates{(2.73E+11,5.57E+02)(9.72E+11,1.88E+03)(3.05E+12,5.31E+03)(8.81E+12,1.42E+04)};                                 
\addplot +[only marks] coordinates{(2.01E+10,7.00E+01)(3.73E+11,1.02E+03)};                 
\addplot +[only marks] coordinates{(1.83E+09,7.62E+00)(4.39E+12,7.25E+03)};
\addplot +[only marks] coordinates{(4.30E+09,1.42E+01)};
\addplot +[only marks] coordinates{(1.93E+09,9.26E+00)(4.62E+12,1.14E+04)};                    
\addplot +[only marks] coordinates{(2.21E+12,2.49E+03)(5.70E+12,5.89E+03)(1.21E+13,1.18E+04)}; 
\addplot +[only marks] coordinates{(7.29E+08,6.50E+00)(2.08E+10,1.11E+02)(7.80E+11,3.85E+03)};
\addplot +[only marks] coordinates{(9.61E+11,1.93E+03)(2.48E+12,4.62E+03)};
\addplot +[only marks] coordinates{(2.10E+08,1.54E+00)(4.62E+08,2.18E+00)(1.02E+09,3.35E+00)};
\addplot +[only marks] coordinates{(1.50E+08,1.55E+00)(2.64E+11,4.84E+02)};
\legend{Cyclic 13--15, Noon18--20, Chandra18--21, Katsura 15--18, Gaukwa 7--8, Vortex 5--6, N-body 5, Gridanti 3--4, Sonic 8--10, Graphmodel 6--8, Eco 20--21, Reimer 13--15, VortexAC 4--5}
\end{loglogaxis}
\end{tikzpicture}
\caption{Measured running time (using 8 cores) against
the invariant $T$ from equation \eqref{main-time1} for several benchmark examples.
Data from table~\ref{table:running} page~\pageref{table:running}.
\label{plot:running}}
\end{figure}
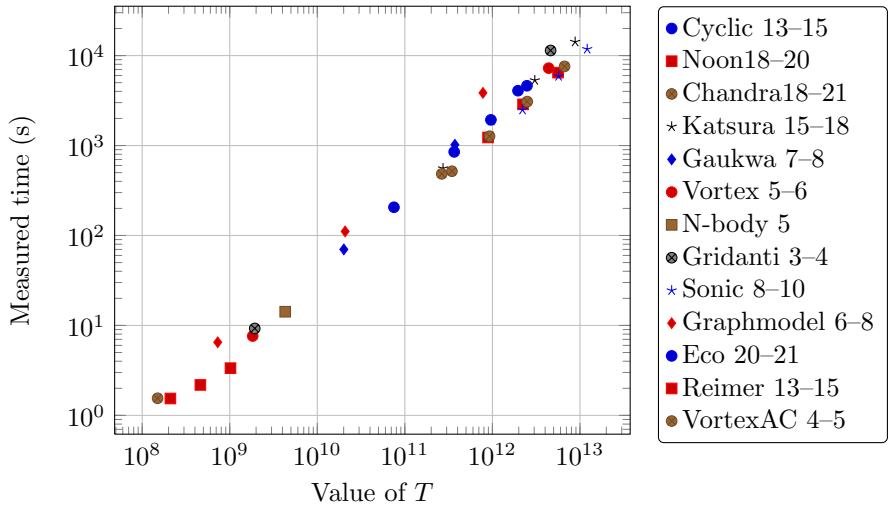

Experiments were performed on a large number of examples,
including some very large benchmark systems (Fig.~\ref{plot:running}). 
It is worth to mention that numerical stability issues did arise.
Those were circumvented by a careful error analysis and a recovery step
in the linear algebra routines (Sec.~\ref{sec:implementation}).
The experiments in Fig~\ref{plot:running} show no noticeable 
running time increase due to the eventual recovery step.

\begin{claim} The algorithm is scalable. 
\end{claim}

Modern computers are built with multiple cores, and serial
complexity analysis does not guarantee a competitive parallel running-time.
The program was successfully tested on a parallel
environment with up to 8 nodes running 8 cores each.
 
When analyzing parallel algorithms, the most important 
complexity invariant is the communication complexity. The parallel 
version of the algorithm will exchange  at most $O(\sum v_d)$ messages 
of size $O(n)$. Again, this bound alone does not imply good practical
scalability properties, so experimentation is necessary.

For each given polynomial system, let $T_N$ be the measured
running time with $N$ cores. In the benchmark families tested,
the running time was of the order of $O(N^{0,93})$.

Another parallel algorithm for the same problem was described by
\ocite{Chen-Lee-Li}. Figure 5 in their paper shows the speedup
factor for the {\em Cyclic-15} benchmark example in a similar
multi-node environment. From their picture, their speed-up factor 
$T_{32}/T_{64}$ 
from 32 to 64 cores is around $1,80$ against $1,96$ obtained here.

\claim{The program performance is comparable to the best available code.}

\rev{Other free software for mixed volume computations are MixedVol by
\ocite{Gao-Li-Mengnien}, DEMICs by \ocite{Mizutani-DEMICs} and PHCpack by
\ocite{Verschelde-795}. \revd{Closed source} programs are MixedVol2.0 by
\ocite{Lee-Li} and MixedVol3.0 by \ocite{Chen-Lee-Li}.}

At this time, \ocite{Lee-Li} and \ocite{Chen-Lee-Li} 
have the best published timings 
for the problem of computing mixed volumes and mixed cells. Since the 
algorithm in this paper is different, the results obtained here
are better for some benchmark families and worse for others.

Overall, the implementation of {\sc AllMixedCellsFull} 
appeared to be reliable for systems with
{\em output size} of around $n V \simeq 10^7$ and beyond. In each 
of the benchmark families tested, the running time grows moderately
with respect to the output size. In some of the benchmark families,
a big performance gain was obtained by using a random path heuristic. 

\medskip
\par
I would like to thank Elizabeth Gross for explaining 
graphical models to me and providing the {\em graphmodel}
example (Table~\ref{table:running}), and Ioannis Emiris for
useful conversations on mixed volume estimation.
I would also like to thank Leonid Gurvits, \revb{Bernd Sturmfels}
and \revc{four} anonymous referees
for their corrections and comments. Special thanks to the NACAD
staff for keeping the computer running despite severe hardware
malfunctions.

\section{Mixed Legendre transform and tropical varieties}

In order to introduce our main tools, it is convenient to work in a more general setting.
Some of the supports $A_i$ may be repeated, and there is some work to save by
considering {\em semi-mixed volumes}, that is mixed volumes with 
multiplicities. Through this paper, $A_1, \dots, A_s$ are
finite subsets of $\mathbb Z^n$, $s\le n$. Multiplicities
$m_1 + \dots + m_s = n$ are fixed, $m_i \ge 0$. The semi-mixed volume 
is defined by
\begin{eqnarray*}
\glossary{$s$ & Number of different supports $A_i$.}
\glossary{$m_i$ & Multiplicity of each support $A_i$.}
V &=& V( \conv{A_1}, m_1; \cdots; \conv{A_s}, m_s)
\\
&\defeq& V( \underbrace{\conv{A_1}, \dots, \conv{A_1}}_{\text{$m_1$ times}}, \dots,  \underbrace{\conv{A_s}, \dots, \conv{A_s}}_{\text{$m_s$ times}}).
\end{eqnarray*}
Thus, $n!V$ is the generic number of roots in $(\mathbb C^{\times})^n$ 
of polynomial systems of the form
\[
f_{ij}(\mathbf x) = 0, \hspace{2em}1 \le j \le m_i,\hspace{2em} 1 \le i \le s, 
\]
with
\[
f_{ij}(\mathbf x) = 
\sum_{\mathbf a \in A_i} f_{ij{\mathbf a}}x_1^{a_1}x_2^{a_2}\cdots x_n^{a_n} 
.
\]

\bigskip
\begin{definition}\label{def:generic}
\revb{A real function} 
$b: A_1 \sqcup \cdots \sqcup A_s \rightarrow \mathbb R$ is in {\em general position} if
and only if, for any subset $S$ of 
$\{ [-\mathrm e_i, \mathbf a, b(i,\mathbf a)] : \mathbf a \in A_i\}
\subset \mathbb R^s \times \mathbb R^n \times \mathbb R$
of cardinality $k \le n+s+1$, either the vectors in $S$ are linearly independent or 
$[0, \cdots, 0,1]$ is a linear
combination of the points in $S$.
\end{definition}

In particular, $b$ is in general position with probability 1, that is
outside of a certain set of Lebesgue measure zero. 
The function $b$ appears in mixed volume computation papers as a {\em random lifting}. 
\medskip
\par
\begin{figure}
\centerline{\resizebox{\textwidth}{!}{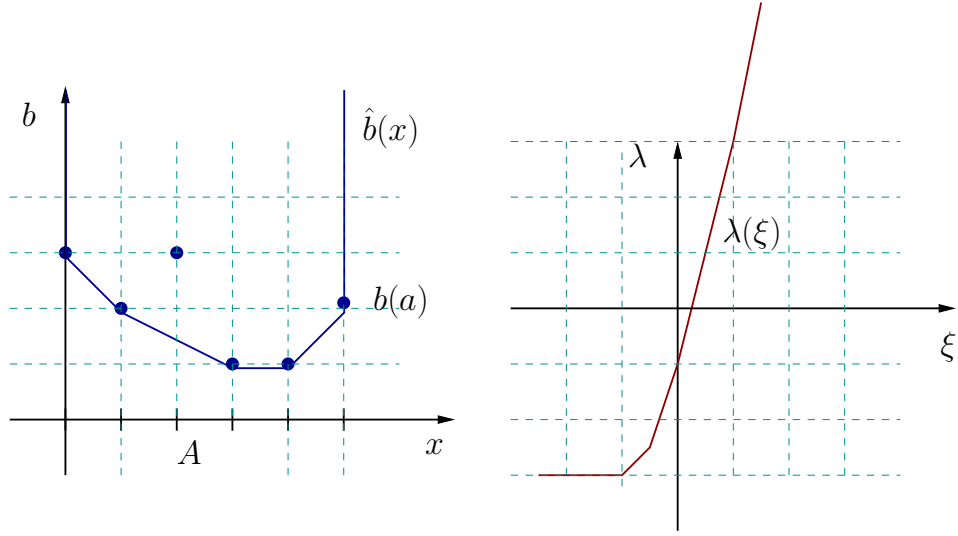}}
\caption{Legendre duality for $s=1$.\label{fig:legendre}}
\end{figure}
Let $b_i = b_{|A_i}$.
\glossary{$b_i = b_i(\mathbf a) = b(i, \mathbf a)$&Lifting value for $\mathbf a \in A_i$. }
The {\em Legendre dual} of $b_i$ is the function
\glossary{$\lambda_i = \lambda_i(\mathbf \xi)$&Legendre dual for the lifting $b_i$. }
$\lambda_i: (\mathbb R^n)^* \rightarrow \mathbb R$ defined by
\[
\lambda_i (\boldsymbol \xi) = \max_{\mathbf a \in A_i} \boldsymbol \xi(\mathbf a) - b(\mathbf a) 
.
\]
The function $\lambda_i$ is convex. Its Legendre dual is the lower
convex hull of $b_i$, defined as 
the largest convex function $\hat b_i: \conv{A_i} \rightarrow 
\mathbb R$ with $\hat b_i(\mathbf a) \le b_i(\mathbf a)$ for all $\mathbf a \in A_i$. 
(Fig.\ref{fig:legendre}). \revd {Its {\em epigraph} $\{(\mathbf x,y):y \ge \hat b(\mathbf x)\}$ can
be seen as the convex hull of the set $\{(\mathbf a,b_i(\mathbf a):\mathbf a \in A_i\}$ and 
a point {\em at infinity} $(0,\infty)$.
The non-vertical \revc{faces} of the epigraph project onto a 
subdivision of the Newton polytope $\conv{A_i}$.}

\revc{\begin{remark}
In the language of tropical algebraic geometry, the Legendre dual $\lambda_i$ 
of $b_i$ is a tropical polynomial. 
\end{remark}}

\revc{Assume that $b$ is in general position.}
To any $\boldsymbol \xi \in (\mathbb R^n)^*$, we associate the numbers
$m_1(\boldsymbol \xi), \dots, m_s(\boldsymbol \xi)$ such that 
\glossary{$m_i(\boldsymbol \xi)$ & Number of times $\lambda_i(\mathbf \xi)$ attained, minus one.}
$\lambda_i(\boldsymbol \xi)$ is attained
for exactly $m_i(\boldsymbol \xi) +1$ values of $\mathbf a \in A_i$. We also associate to
the pair $(i,\boldsymbol \xi)$ a \revc{face} $L_{i,\boldsymbol \xi}$ of $\mathrm{Graph}(\hat b_i)$, 
\begin{equation}
L_{i,\boldsymbol \xi} = \left\{ (\mathbf x, \hat b(\mathbf x)): 
\mathbf x \in \conv{A_i}
\text{ and }
\lambda_i(\boldsymbol \xi) = \boldsymbol \xi(\mathbf x) - \hat b(\mathbf x)
\right\}
.
\glossary{$L_{i,\boldsymbol \xi}$ & Facet of $\mathrm{Graph}(\hat b_i)$.}
\end{equation}
\medskip
\par

Let $t_1, \dots, t_s > 0$ be indeterminates. We consider now the mixed
lifting $t_1 b_1 + \cdots + t_s b_s$ of 
\revd{the set of formal linear combinations}
$t_1 A_1 + \cdots + t_s A_s$
\revd{ by:}
\[
t_1 \mathbf a_1 + \cdots + t_s \mathbf a_s
\mapsto
t_1 b_1(\mathbf a_1) + \cdots + t_s b_s(\mathbf a_s).
\]
\revc{No ordering between the $t_i$ is assumed. Yet, to} every $\mathbf \xi \in \mathbb R^n$, we \revc{can} associate a 
\revc{face}
$L_{\mathbf \xi}$
of ${Graph}(t_1 \hat b_1 + \cdots \reva{+}t_s \hat b_s)$,
\begin{equation}
\label{eqlowerfacet}
L_{\boldsymbol \xi} = t_1 L_{1,\boldsymbol \xi} + \cdots + t_s L_{s,\boldsymbol \xi}. 
\glossary{$L_{\boldsymbol \xi}$ & Facet of $\mathrm{Graph}(\sum t_i \hat b_i)$,
$t_i$ indeterminates.}
\end{equation}
\revc{The \revc{face} $L_{\boldsymbol \xi}$ is is well defined because
the $L_{i,\boldsymbol \xi}$ are independent with respect to the specialization 
of the variables $t_i$.} This \revc{face} is
said to be of type $(m_1(\boldsymbol \xi), \dots, m_s(\boldsymbol \xi))$. 
\revc{Since $b$ is in general position, the dimension of $L_{\boldsymbol \xi}$} is
$m_1(\boldsymbol \xi)+\cdots+m_s(\boldsymbol \xi)$. 

Reciprocally, any $n$-dimensional lower facet $L$
admits a unique vector $\boldsymbol \xi \in (\mathbb R^n)^*$ with 
$L=L_{\boldsymbol \xi}$. More generally, let $L$ be a lower \revc{face} of
any dimension and define $\Xi(L) = \{ \xi: L=L_{\boldsymbol \xi}\}$. Then
$\Xi(L)$ is a (possibly unbounded) \revb{polyhedron} of dimension $n-\sum m_i(\boldsymbol \xi)$, `dual' to $L$.
\glossary{$\Xi(L)$& Possibly unbounded \revb{polyhedron} dual to \revc{face} $L$.} 

\begin{definition}\label{def:mixed}
\revb{Let $m_1, \dots, m_s$ be the fixed 
multiplicities associated to each polytope,
$m_1 + \cdots + m_s=n$.}
A {\em lower mixed facet} is a lower \revc{face} of type $(m_1, \dots, m_s)$.
A {\em mixed cell} is the projection of a lower mixed facet into $\mathbb R^n$.
\end{definition}

The mixed volume $V$ is equal to the sum of the volume of the mixed cells.
The basic idea for enumerating the mixed facets is to explore certain tropical varieties.
\revc{One can specify a tropical variety in $\mathbb R^n$
by bounding the $m_i(\boldsymbol \xi)$ below. For instance,
\[
T_i \defeq \{ \boldsymbol \xi \in \mathbb R^n: m_i(\boldsymbol \xi) \ge 1 
\}
=
\{ \boldsymbol \xi \in \mathbb R^n: \max( \mathbf a \boldsymbol \xi - b_i(\mathbf a)) \text{ attained twice}\}
\]
is the hypersurface defined by the tropical polynomial
\[
\sum_{\mathbf a \in A_i} (-b_i(\mathbf a)) \xi_1^{a_1} \xi_2^{a_2} \cdots \xi_n^{a_n}.
\]
}
\revc{Since the lifting $b_i$ is assumed to be in general position, the set
\[
\{ \boldsymbol \xi \in \mathbb R^n: m_i(\boldsymbol \xi) \ge 2 
\}
\]
is a codimension 2 hypersurface, which is the stable intersection
$2T_i \defeq T_i \cap_{\text{st}} T_i$ \cite{Maclagan-Sturmfels}*{Ch.3}.
More generally if $1 \le m \le n$,
$\{ \boldsymbol \xi \in \mathbb R^n: m_i(\boldsymbol \xi) \ge m 
\}$ is the stable intersection 
$m T_i =  T_i \cap_{\text{st}} \cdots \cap_{\text{st}} T_i$.
The lower mixed facets are the $L_{\boldsymbol \xi}$ where 
$\boldsymbol \xi \in X_n$ and $X_n$ is the point configuration
\[
X_n \defeq m_1 T_1 
\cap_{\text{st}}
\cdots
\cap_{\text{st}}
m_s T_s
= 
\{ \boldsymbol \xi \in \mathbb R^n: m_1(\boldsymbol \xi) \ge m_1,
\dots, m_s(\boldsymbol \xi) \ge m_s
\}
.\]
} 

\reva{In order to find $X_n$ we will proceed by induction on the dimension
$0 \le d \le n$. At each step we will explore a one-dimensional tropical
variety $\revc{G}_{d}$ containing $X_d$. Those varieties need to satisfy certain
genericity hypotheses, so we proceed as follows:}

Let $F_0 \subset F_1 \subset \cdots \subset F_n=\mathbb R^n$ be a \revb{flag
of generic {\em affine} subspaces}.\glossary{$F_0 \subset \cdots \subset F_n$ & Generic affine flag in $\mathbb R^n$.}\glossary{$\mm{i}{d}$ & Certain non-decreasing sequence.} \revb{At dimension $d$, we will produce $X_d \subset F_d$ 
corresponding to certain 
multiplicities $\mm{1}{d}, \dots, \mm{s}{d}$. 
If $d<n$, we
will then explore $\revc{G}_{d}$ and produce $X_{d+1}$.}
For \rev{$i \in \{1,\dots, s\}$ and $d \in \{0, \dots, n\}$}, 
\revc{we choose the multiplicities} $\reva{0=\mm{i}{0} \le \mm{i}{1} \le \cdots \mm{i}{\revb{n}}=m_i} \in
\mathbb N_0$ \revc{so that 
$\sum_i \mm{i}{d}=d$. To do this, we start at $d=0$ with all the $\mm{i}{0}=0$ and
then increase exactly one of the $\mm{i}{d}$ at each step $1 \le d \le n$.}  
\reva{We define:
\begin{eqnarray*}
X_d &=&  
F_d \cap \left(\mm{1}{d} T_1 
\cap_{\text{st}}
\cdots
\cap_{\text{st}}
\mm{s}{d} T_s \right)
\\
\revc{G}_d &=& 
F_d \cap \left(\mm{1}{d-1} T_1 
\cap_{\text{st}}
\cdots
\cap_{\text{st}}
\mm{s}{d-1} T_s \right).
\end{eqnarray*}
Using explicit notation,
}
\begin{eqnarray*}
X_d &=&  \{ \boldsymbol \xi \in F_d:\forall i,\  m_i(\boldsymbol \xi) \ge \mm{i}{d} \} \\
\revc{G}_d &=& \{ \boldsymbol \xi \in F_d:\forall i,\  m_i(\boldsymbol \xi) \ge \mm{i}{d-1} \}.
\glossary{$X_d$  & Certain zero-dimensional tropical variety.}
\glossary{$\revc{G}_d$ & Certain one-dimensional tropical variety.}
\end{eqnarray*}
\reva{The induction starts with $X_0 = \{F_0\}$. The induction step is 
possible because of the result below, stated in classical terms:}

\begin{theorem} \label{th:connected}
Assume that $V( \conv{A_1}, m_1; \cdots; \conv{A_s}, m_s) \ge 1$. 
If $b: A_1 \sqcup \cdots \sqcup A_n \rightarrow \mathbb R$ is in general position and
the flag $F_d$ generic, then
\begin{enumerate}[(a)]
\item Each set $\revc{G}_d$, $d=1, \dots, n$, is a finite closed union of line segments and half-lines. 
\item $X_{d-1} = \revc{G}_d \cap F_{d-1}$.
\item Each connected component of $\revc{G}_d$ \revc{intersects $F_{d-1}$ at least in one point}.
\item (Transversality) All points in $\revc{G}_d \cap F_{d-1}$ are in the interior of a line
segment or a half-line of $\revc{G}_d$.
\end{enumerate}
\end{theorem}

Moreover, a certain balancing condition (Lemma~\ref{lem:balancing}) holds for the edges incident to a vertex.

The
{\em structure theorem} of Tropical Algebraic Geometry \cites{Cartwright-Payne, Maclagan-Sturmfels}
guarantees the connectedness
of one-dimensional \revc{intersections of} tropical varieties $\revc{G}=\revc{G}_d$, as long as the 
underlying complex variety is irreducible.
Example~\ref{counterexample} shows how $\revc{G}_{n}$ can fail to be connected on a positive measure set of liftings $b$ \reva{when the underlying complex variety is
reducible}. \revb{The precise combinatorial 
conditions for the connectedness of tropical hypersurfaces are described
by \ocite{Yu}.}

\medskip
\par

The main tool for proving Theorem~\ref{th:connected} is the mixed
Legendre transform 
$\boldsymbol\xi \mapsto \lambda_1(\boldsymbol\xi), \dots, \lambda_s(\boldsymbol\xi)$. As before,
the $t_1$, \dots, $t_s$ are positive indeterminates. The `epigraph' of
the function $t_1 b_1 + \cdots + t_s b_s$ is the set of all the pairs
$(\boldsymbol \xi, t_1 \lambda_1 + \cdots + t_s \lambda_s)$ 
so that $(\boldsymbol \lambda, \boldsymbol \xi)$ is a solution
of the system of inequalities
\begin{equation}\label{inequalities}
\mathbf C
\left[
\begin{matrix}
\boldsymbol\lambda \\ \boldsymbol\xi 
\end{matrix}
\right]
\le
\mathbf b
\end{equation}
where $\mathbf C$ is the {\em Cayley matrix}
\begin{equation}\label{big-A}
\mathbf C = \bigmatrix 
\glossary{$\mathbf C$, $\mathbf b$ & Cayley matrix and lifting vector.}
\end{equation}
and each support $A_i$ is represented by a matrix with rows $\boldsymbol a \in A_i$. As no confusion
can arise, we use the same symbol for the support and its representing matrix.

A lower \revc{face} $L$ of dimension $d$ can be represented by a pair 
$(\boldsymbol \lambda(\boldsymbol \xi),\boldsymbol \xi)$ so that $L=L(\boldsymbol \xi)$,
with exactly $s+d$ equalities in \eqref{inequalities}. In the language
of linear programming, those are known as the {\em active constraints} while the
strict inequations are deemed {\em inactive}. 

The same lower \revc{face} $L$ can also be represented by its set of active
constraints.  
Its dual $\{\boldsymbol \xi\}=\Xi(L)$ is the set of solutions
of
\begin{align*}
[-\mathrm e_i, \mathbf a ] 
    \left[\begin{matrix}\boldsymbol \lambda(\mathbf \xi) \\ \boldsymbol \xi \end{matrix}\right] &= b(i, \mathbf a) &&
\text{for $(i, \mathbf a)$ an active constraint},
\\
[-\mathrm e_i, \mathbf a ] 
    \left[\begin{matrix}\boldsymbol \lambda(\mathbf \xi) \\ \boldsymbol \xi \end{matrix}\right] &< b(i, \mathbf a) &&
\text{for $(i, \mathbf a)$ inactive.}
\end{align*}

The tropical varieties $\revc{G}_{d}$ can be explored by {\em pivoting} from each $d$-dimensional
lower \revc{face} to its neighboring \revc{faces}. 
The numerics for pivoting are a fallback from the techniques
of the simplex algorithm. This is explained in section~\ref{sec:pivoting}.

\begin{remark}\label{predecessors}
The other known practically efficient algorithms are those 
by ~\ocite{Mizutani-Takeda-Kojima}
and ~\cite{Lee-Li}. Both algorithms function by enumerating certain viable lower \revc{faces}
of dimension $k$, for $0 \le k \le n$. 
To simplify the discussion we assume $n=s$ and $m_1 = \cdots = m_n = 1$. 
In the language of this paper, those algorithms enumerate lower \revc{faces}
corresponding to solutions of \eqref{inequalities} with exactly $2$ equalities for
the $i$-th block of $\mathbf C$, only for $i$ in a cardinality $k$ subset of $\{1, \dots, n\}.$
Given one of such \revc{faces} for $k<n$, they extend it to a $k+1$ \revc{face} using ideas
of linear programming. It can happen that some \revc{faces} are not extensible, and a lot of
effort is made to devise heuristics that prune the decision tree as early as possible. No complexity analysis is available.
\end{remark}

\begin{figure}
\centerline{\resizebox{\textwidth}{!}{\includegraphics{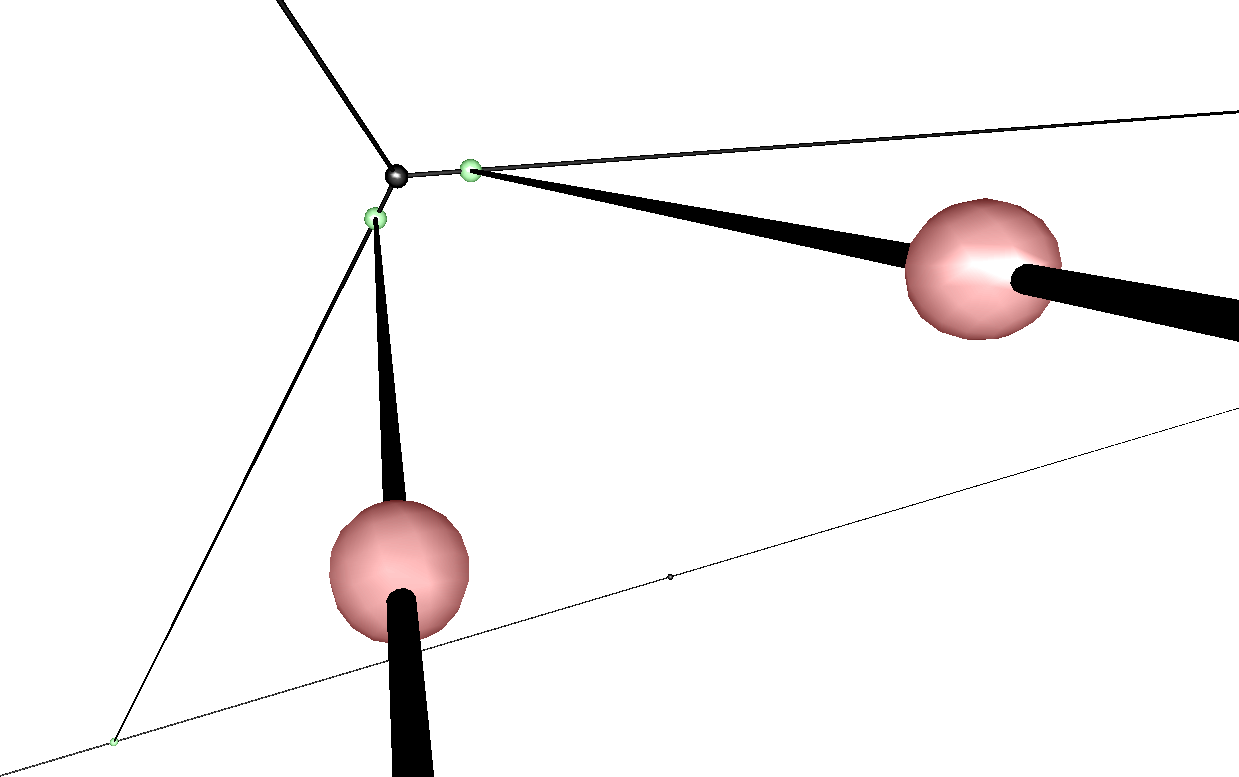}}}
\caption{\label{fig-cyclic3}Cyclic-3 polynomial system: $X_3$ is represented by the two big balls, $\revc{G}_3$ by
the two thick lines going through $X_3$. $\revc{G}_3$ cuts $F_2$ on $X_2$ (light green balls). 
\revd{$G_2$ is visible as three half-lines starting at the 
black ball afar. Albeit at infinity, $F_2$ was artificially represented
as a usual plane. This is why the lines in $G_3$ `stop' at $F_2$.
Similarly, $F_1$ is represented as a usual line, and parts of $G_1$
are visible as the thin line from bottom left to right.} 
This figure was obtained for
$\mm{1}{1}=\mm{2}{2}=\mm{3}{3}=1$.}
\end{figure}
\begin{example}\label{counterexample}(Fig.\ref{fig-cyclic3}).
Consider the {\em Cyclic-3} polynomial system, and replace the coefficients with random coefficients:
\begin{eqnarray*}
c_0 x_1 + c_1 x_2 + c_2 x_3 &=& 0\\
c_3 x_1 x_2 + c_4 x_2 x_3 + c_5 x_3 x_1 &=& 0\\
c_6 x_1 x_2 x_3 +c_7 &=& 0
\end{eqnarray*}
The zero set of the first two equations is reducible: eliminate $x_3$ from the first equation
and substitute in the second to obtain an equation of the form $A x_1^2 + 2B x_1x_2 + C x_2^2$,
which clearly factors. 
Let $\mm{1}{1}=\mm{2}{2}=\mm{3}{3}=1$ as in the figure.

The figure show that for the following values of $b_i\revb{=c_i}$, the tropical variety
$\revc{G}_3$ is disconnected \revc{for a generic flag}:

\centerline{\begin{tabular}{|cl|cl|cl|}
\hline
$b_0 =$& $0.0681718062929322 $ & $b_3 =$& $0.8654168322306781$ &  $b_6 =$ & $0.6575801418616753$ \\
$b_1 =$& $0.2764482146232536 $ & $b_4 =$& $0.6630347993316177$ &  $b_7 =$ & $0.2139433513437121$\\
$b_2 =$& $0.4266688073141105 $ & $b_5 =$& $0.2369372029023467$ & & \\
\hline
\end{tabular}}

The mixed vertices of $X_3$ are:
\[
\begin{split}
X_3 = \left\{
\left[\begin{matrix}
1.7041246197535198...10^{-01}\\-2.5568513445391900...10^{-1}\\5.2890946299653028...10^{-01}
\end{matrix} \right]
,
\left[\begin{matrix}2.8794667050532442...10^{-1}\\4.9622307883564581...10^{-1}\\-3.4053295882300698...10^{-1}
\end{matrix} \right]
\right\}
\end{split}
\]

\end{example}

\section{Proof of Theorem~\ref{th:connected}}

A {\em vertex} in $\revc{G}_d$ is a point $\boldsymbol \xi \in F_d$ such that
$\sum m_i(\boldsymbol \xi)=d$.
An {\em edge} in $\revc{G}_d$ is a one-dimensional
intersection $F_d \cap L_{\boldsymbol \xi}$
where $L_{\boldsymbol \xi}$ is the lower \revc{face} associated to
a vector $\mathbf \xi$ as in \eqref{eqlowerfacet}.
Equivalently, an edge  is the projection onto $\boldsymbol \xi$-space
of a non-empty solution
set of \eqref{inequalities} in $F_d$ with prescribed $\mm{1}{d-1}+1, \cdots, \mm{s}{d-1}+1$ equalities
in the respective block, and no more equalities.

The set $\revc{G}_d$ is a finite union of vertices
and edges. Edges may be bounded or unbounded. Bounded edges are open
segments, whose \revb{end points} are vertices of $\revc{G}_d$. Unbounded edges
cannot be a line, for otherwise all the $A_i$'s would be contained
in an hyperplane orthogonal to that line. Therefore, unbounded edges
are half-lines, bounded in one side by a vertex of $X_0$. (Fig.\ref{fig:Cohn3})
Recall that $m_i(\boldsymbol \xi)$ is the number of 
values of $\mathbf a \in A_i$ such that 
$\lambda_i(\boldsymbol \xi) = \mathbf a \boldsymbol \xi - b(i, \mathbf a)$.

\begin{figure}
\centerline{\resizebox{\textwidth}{!}{\includegraphics{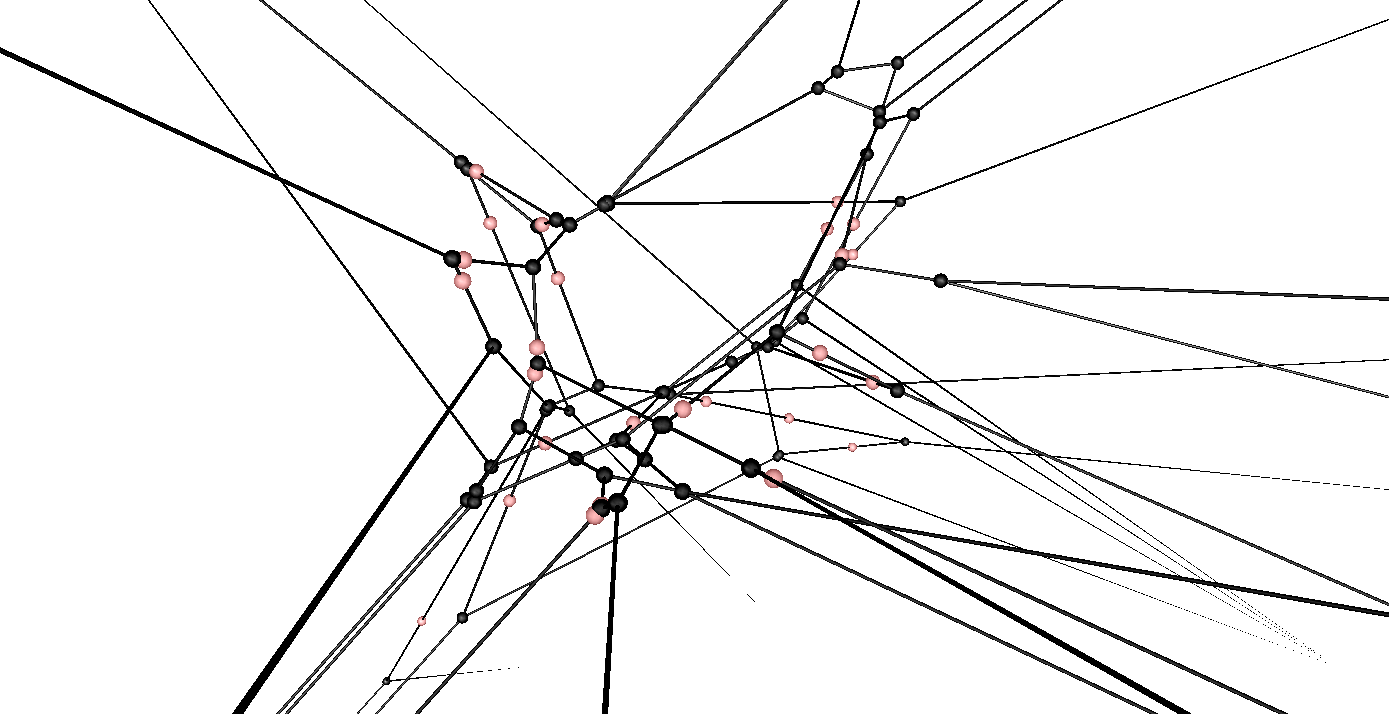}}}
\caption{\label{fig:Cohn3}
The Cohn 3 polynomial system from Posso. $X_4$ is pictured as light balls
and the rest of $\revc{G}_4$ as dark balls and lines. 
}
\end{figure}

\begin{lemma}\label{lem:balancing}
Let $\boldsymbol \xi$ be a vertex of $\revc{G}_d$ 
and let $1 \le q \le s$ be the unique integer with $m_q(\boldsymbol \xi) = \mm{q}{d-1}+1$.
\glossary{$q$& Last polytope so that $\mm{q}{d}$ increased at time $d$.} 
Then, there are precisely $m_q(\boldsymbol \xi)+1=\mm{q}{d-1}+2$ 
edges of $\revc{G}_d$ with endpoint $\boldsymbol \xi$.
Moreover, those edges are of the form $\{\boldsymbol \xi + t \Delta_j \boldsymbol \xi: t \in (0, I_j)\}$ where
$I_j$ is either a strictly positive real number, or infinity.
\glossary{$\Delta_j$& Pivoting direction, in $\xi$-space.} 
\glossary{$I_j$& Pivoting distance.} 
The following {\em balancing condition} holds:
\begin{equation}\label{eq:balancing}
\sum \Delta_j \boldsymbol \xi = 0.
\glossary{$\Delta_j \boldsymbol \xi$, $\Delta_j \boldsymbol \lambda$&Pivoting vectors while dropping constraint $j$.}
\end{equation}
\end{lemma}
\revb{This follows directly from the {\em balancing condition} in 
tropical algebraic geometry. For the benefit of general readers, a 
self-contained proof is given below.}

\begin{proof}
Let $\Aact$, $\bact$
\glossary{$\Aact$, $\bact$&Matrix and vector of active constraints.}
be the submatrix of the Cayley matrix (resp. subvector of $\mathbf b$)
with rows in the set of active constraints of 
vertex $\boldsymbol \xi$, plus the $n-d$ affine constraints in $\xi$ that define $F_d$. So $\Aact$ is a $(s+n) \times (s+n)$ matrix and $\bact \in \mathbb R^{s+n}$. 
Let $\Ainact$
(resp. $\binact$)
\glossary{$\Ainact$, $\binact$&Matrix and vector of inactive constraints.}
be \revb{the} submatrix (resp. subvector) of inactive constraints associated to vertex
$\boldsymbol \xi$.
Therefore we assumed that
\[
\Aact \left[\begin{matrix}\boldsymbol \lambda \\ \boldsymbol \xi \end{matrix} \right] = \bact,
\hspace{2em},
\Ainact \left[\begin{matrix}\boldsymbol \lambda \\ \boldsymbol \xi \end{matrix} \right]< \binact
.
\]

The edges incident to $\boldsymbol \xi$ are obtained by releasing
one of the $m_q(\boldsymbol \xi)+1$ equalities in the $q$-th block.

Suppose we release the $j$-th equality, where the index $j$ corresponds to
the row in $\Aact$ associated to that equality.
In particular $j \in J_q=\{ j: 1 \le j - (\sum_{i<q} \mm{i}{d-1}) - q \le \mm{i}{d-1}+2 \}$.
We obtain a system of equalities of the form
\[
\Aact 
\left[ 
\begin{matrix} \boldsymbol \lambda + t\Delta_j \boldsymbol \lambda \\
\boldsymbol \xi + t\Delta_j \boldsymbol \xi \end{matrix} \right]
= \bact - t\mathrm e_j
\]
where $\mathrm e_j$ is the $j$-th canonical basis vector and $t>0$ is
indeterminate. This simplifies to
\[
\Aact 
\left[ 
\begin{matrix} \Delta_j \boldsymbol \lambda \\
\Delta_j \boldsymbol \xi \end{matrix} \right]
= - \mathrm e_j
.
\]

We can solve and find
\[
\left[ 
\begin{matrix} \Delta_j \boldsymbol \lambda \\
\Delta_j \boldsymbol \xi \end{matrix} \right]
= 
- \Aact^{-1} \mathrm e_j
.
\]

The balancing condition follows from the fact that
\[
\Aact 
\left[
\begin{matrix}
\mathrm e_q \\
0
\end{matrix}
\right]
=
- \sum_{j\in J_q} \mathrm e_j .
\]
Multiplying by $\Aact^{-1}$,
\[
\left[
\begin{matrix}
\mathrm e_q \\
0
\end{matrix}
\right]
= -\sum_j \Aact^{-1} \mathrm e_j 
=
\left[
\begin{matrix}
\sum_{j\in J_q} 
\Delta_j \boldsymbol \lambda \\
\sum_{j\in J_q} 
\Delta_j \boldsymbol \xi 
\end{matrix}
\right]
.\]
\revb{
From this we deduce that $\sum_{j\in J_q} 
\Delta_j \boldsymbol \xi =0$.}
\end{proof}

An immediate consequence of \eqref{eq:balancing} is the following Lemma,
to be used in the proof of Theorem~\ref{th:connected}(c).

\begin{lemma}\label{lem:objective} Let $\mathbf Q \in (\mathbb R^n)^*$ be
an arbitrary objective function. Then either $\mathbf Q\ \Delta_j \boldsymbol \xi > 0$ for some
$j \in J_q$, or $\mathbf Q\ \Delta_j \boldsymbol \xi = 0$ for all $j \in J_q$.
\end{lemma}

In order to pivot from \revc{face} to \revc{face}, we need to know the value
of $I_j$ in Lemma~\ref{lem:balancing}.

\begin{lemma}\label{lem:Ij}
In the conditions above,
\[
I_j = \min
\left(
\left\{
+\infty
\right\}
\cup
\left\{ \rule{0em}{2.4ex}
t(i,\mathbf a): \mathbf a \in A_i, \lambda_i(\boldsymbol \xi)<\lambda_i \text{ and } t(i,\mathbf a) > 0 \right\}
\right)
\]
where
\[
t(i,\mathbf a) = \frac{ [-\mathrm e_i, \mathbf a]
\left[ \begin{matrix} \lambda \\ \xi \end{matrix} \right] 
-
b(i,\mathbf a)
}
{
[-\mathrm e_i,\mathbf a]  \Aact^{-1} \mathrm e_j
}
.
\glossary{$t(i,\mathbf a)$&score of inactive constraint $[i, \mathbf a]$}
\]
\end{lemma}

\begin{proof}
In order to find $I_j$, we solve
\[
\Ainact 
\left[
\begin{matrix}
\boldsymbol \lambda+t\Delta_j \boldsymbol \lambda \\
\boldsymbol \xi + t \Delta_j \boldsymbol \xi 
\end{matrix}
\right]
\le
\binact
\]
with exactly one equality.
This is the same as
\[
\Ainact 
\left[
\begin{matrix}
\boldsymbol \lambda\\
\boldsymbol \xi  
\end{matrix}
\right]
-
\binact
\le
t \Ainact \Aact^{-1} \mathrm e_j
\]
with $t>0$ and exactly one equality. The left hand side is
always negative. 
For each inactive constraint $[-\mathrm e_i, \mathbf a]$,
$\mathbf a \in A_i$, set
\[
t(i,\mathbf a)= \frac{ [-\mathrm e_i, \mathbf a] \left[
\begin{matrix}
\boldsymbol \lambda\\
\boldsymbol \xi  
\end{matrix}
\right]
-
b(i,\mathbf a)
}
{
[-\mathrm e_i,\mathbf a]  \Aact^{-1} \mathrm e_j
}
.
\]

Then $I_j$ is the minimal positive value of $t(i,\mathbf a)$ where $[\mathrm e_i,\mathbf a]$ is
an inactive constraint. In case the set of positive values is empty,
$I_j = +\infty$.
\end{proof}

\begin{proof}[Proof of Theorem~\ref{th:connected}]
We already checked (a), and (b) holds by construction. We prove (c) now.
Let $\mathbf Q_d$ be a non-zero normal vector to $F_{d-1} \subseteq F_d$ so that $F_{d-1}
= \{\boldsymbol \xi \in F_d: \mathbf Q_d \boldsymbol \xi = r_d\}$.
Each connected component
of $\revc{G}(d)$ has finitely many vertices. There is a finite number of possible matrices $\Aact$.
The first $s+d$ rows of each $\Aact$ are constraints and the remaining $n-d$ rows are 
obtained from constraints $\mathbb Q_{d+1}, \cdots, \mathbb Q_n$.
Because the flag $F_0 \subset F_1 \subset \cdots \subset F_n$ is generic, $\mathbf Q_d$ is not orthogonal 
to any of the columns of any $\Aact^{-1}$. Thus, at all vertices of $\revc{G}(d)$, 
$\mathbf Q_d \Delta_j \boldsymbol \xi \ne 0$ and $\mathbf Q_d \Delta_j \boldsymbol \xi$ assumes both strictly positive and negative values.
As there is a finite number of vertices in each connected component $C$ of $\revc{G}(d)$, at least one of the strictly positive (resp. strictly negative)
values corresponds to a half-line. 
Hence, the connected component $C$ has points 
with $\mathbf Q_d \boldsymbol \xi \rightarrow -\infty$ and $\mathbf Q_d \boldsymbol \xi \rightarrow +\infty$. 

By the intermediate value theorem, $C$ must
cut $F_{d-1} = \{ \boldsymbol \xi: \mathbf Q_d \boldsymbol \xi = r_d \}$ at least once. This
proves (c). The transversality condition (d) follows from
the genericity of $r_d$. 
\end{proof}

\begin{remark}
If one picks 
\revb{$r_n \gg r_{n-1} \gg \cdots \gg r_{d+1}$}
then the transversality condition still holds. 
\end{remark}

\section{Facet pivoting}
\label{sec:pivoting}

In this section we produce the equations for pivoting from a point
of $X_d$ to its neighbors in $\revc{G}_d$ \revb{(Lemma \ref{lem:Ij1})} 
and also to pivot \revb{between} $X_d$ \revb{and}
$X_{d+1}$ \revb{(Lemmas \ref{lem:Ij2} and \ref{lem:Ij3})}. We start with a well-known Lemma that can be used to relate
the matrices of active constraints in two adjacent vertices of $\revc{G}_d$.

\begin{lemma}[Rank 1 updates]\label{rank_one_updates} Let $A$ and $B$ be $n \times n$ matrices with $B = A^{-1}$. 
Let $u, v \in \mathbb R^n$. Then $A - uv^T$ is invertible if and only if $v^TBu \ne 0$. Moreover, if $A - uv^T$ is
invertible,
\[
(A - u v^T) ^{-1} = B + \frac{1}{1-v^TBu} Buv^TB .
\]
\end{lemma}
\begin{proof}
First of all, notice that $\det (A-uv^T) = \det(A) \det(I-A^{-1}u v^T) = 1-v^T A^{-1} u$.
Assuming this is different from zero,
\[
(A - uv^T)^{-1} = 
B (I - uv^TB)^{-1}  
=
B (I + \frac{1}{1-v^TBu} uv^TB)
.
\]
The last equality follows from multiplying $I-uv^TB$ and $I + \frac{1}{1-v^TBu} uv^TB$.
\end{proof}

\begin{figure}
\centerline{\resizebox{\textwidth}{!}{\includegraphics{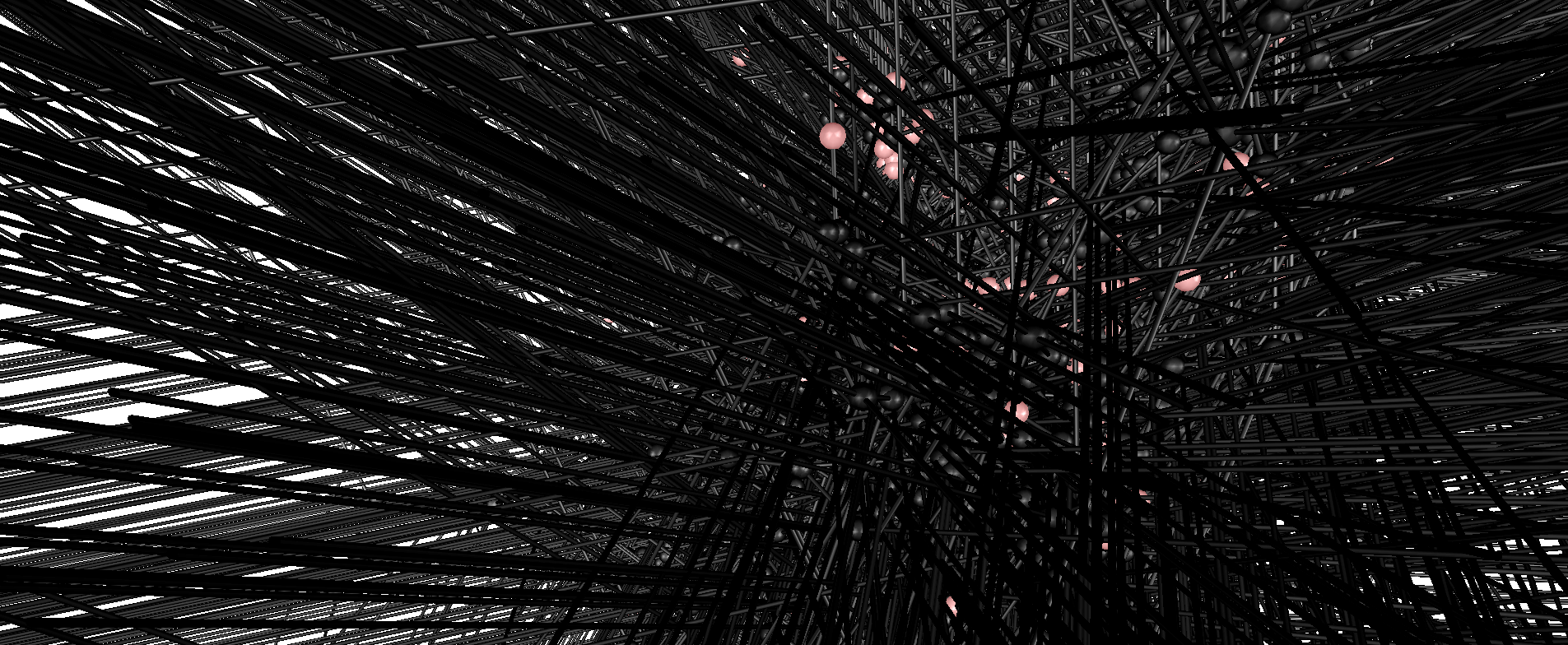}}}
\centerline{\resizebox{\textwidth}{!}{\includegraphics{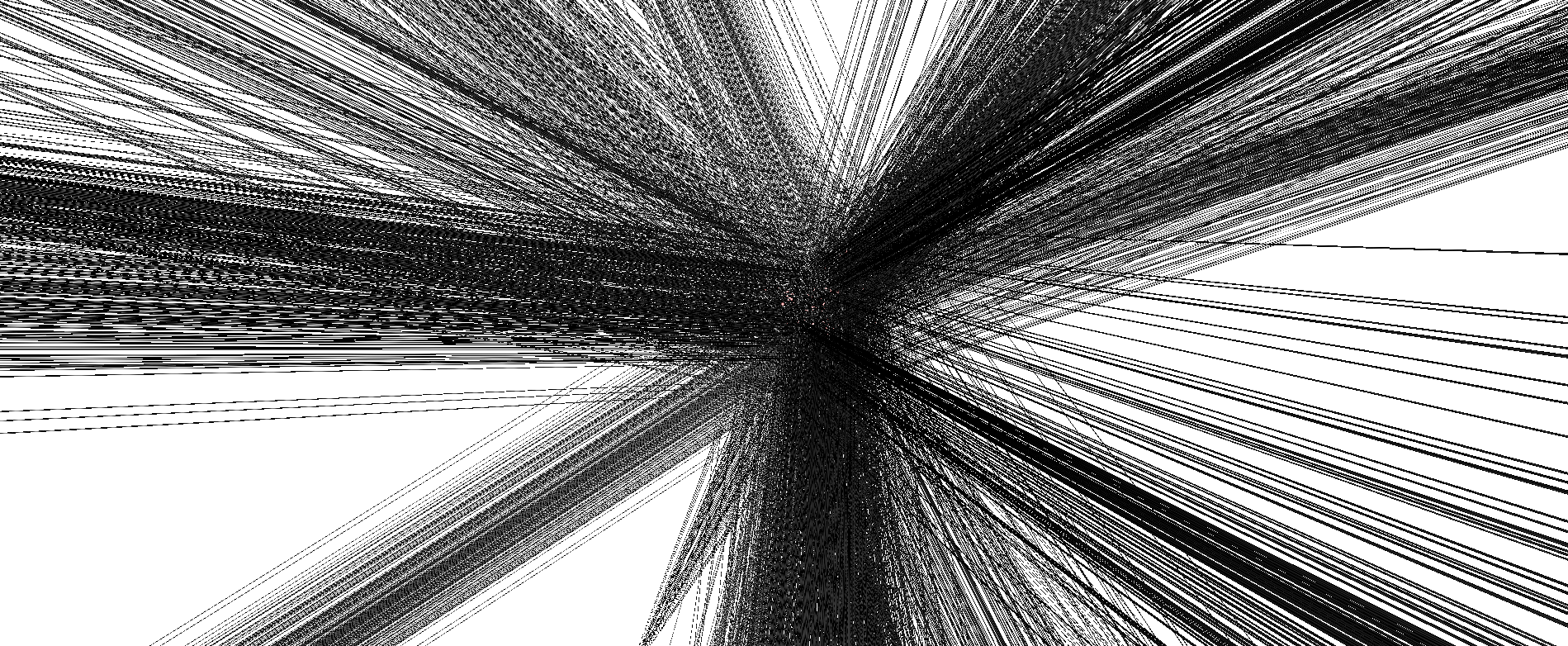}}}
\caption{\label{fig:infinity}
Those pictures illustrate the reason we prefer to take an affine
flag `at infinity'.
The top picture shows $\revc{G}_{18}$ for the 
sonic3 polynomial system (18 equations in 18 variables) viewed 
from close to the
origin. The pipes that are interrupted are actually crossing the
screen. The bottom picture shows $\revc{G}_{18}$ viewed from far away.
}
\end{figure}

Now, assume that $F_{d}\subset F_{d+1}$ is defined by the equation $Q_{d+1} \xi = R^{n-d} r_{d+1}$, where $r_{d+1} > 0$ and 
$R>0$ is a parameter that
will tend to infinity. The reason for the choice $R \rightarrow \infty$
comes from practical considerations.

Figure~\ref{fig:infinity} shows a typical graph $\revc{G}_d$. The region
close to the origin seems overcrowded with interlaced edges, while
the `spikes' do not appear to be as numerous as the finite edges. 
Cutting by a plane at infinity minimizes the number of intersections,
hence the number of \revc{faces} to be found. 
There are other advantages related to the stability of the numerical
implementation that will be discussed  
in section~\ref{sec:implementation}.

We write down below the pivoting equations for $\revc{G}_d$.
Let $\{\xi\}$ be a vertex of $\revc{G}_d$ and $\lambda=\lambda(\xi)$.
This means that there are $s+d$ active constraints such that
\begin{equation}\label{explicit-Aact}
\Aact
=
\left[
\begin{matrix}
-\mathrm{e}_{i_1}     & a_{1} \\
\vdots           & \vdots \\
-\mathrm{e}_{i_{d}} & a_{d} \\
0                & Q_{d+1} \\
\vdots & \vdots \\
0                & Q_{n} \\
\end{matrix}
\right]
\hspace{1em}
\text{and}
\hspace{1em}
\bact = 
\left[
\begin{matrix}
b(i_1, a_1) \\
\vdots\\
b(i_{d},a_{d}) \\
R^{n-d} r_{d+1} \\
\vdots\\
R r_{n} 
\end{matrix}
\right]
\end{equation}
are respectively the matrix of active constraints and the vector of active constraints.
There are at least $\mm{i}{d-1}+1$ occurrences of $\mathrm{e}_i$. There is a unique
$1 \le q \le s$ so that there are $\mm{q}{d-1}+2$ occurrences of  $\mathrm{e}_q$,
and those are the active constraints that may be `dropped'.
Facets are uniquely defined by the set of active constraints:

\begin{lemma}\label{lem:genericity}
Let $R \ge 0$.
Assume that 
\[
\Aact
\left[
\begin{matrix}
\lambda \\
\xi \\
\end{matrix}
\right]
=
\bact
.
\]
Then $\Aact$ is invertible. Moreover,
there cannot be any extra $(i,a)$ so that
\[
[\begin{matrix} -\mathrm{e_i}, & a\end{matrix} ] \left[
\begin{matrix}
\lambda \\
\xi \\
\end{matrix}
\right]
=
\mathbf b (i,\mathbf a)
\]
\end{lemma}

\begin{proof}
The matrix $\Aact$ is of the form
\[
\Aact = 
\left[
\begin{matrix}
\mathrm{e}_{i_1} & \mathbf a_1 \\
\vdots & \vdots \\
\mathrm{e}_{i_{d}} & \mathbf a_d \\
0 & Q_{d+1} \\
\vdots & \vdots \\
0 & Q_n 
\end{matrix}
\right]
\]

Because the flag $F_0 \subset \cdots \subset F_n$ is generic and
the $A_i$ are finite,
the span of $Q_{d+1}, \dots, Q_n$ is transversal to any
space spanned by exactly $d$ vectors of the form $a - a'$, where $a, a' \in \cup A_i$. 
After elimination and some row permutations, matrix $\Aact$ factors:
\[
\Aact = P L 
\left[
\begin{matrix}
-I & U_{12} \\
    &U_{22} \\
    & Q_{d+1} \\
& \vdots \\
& Q_n
\end{matrix}
\right].
\]
$P$ is a permutation. $L$ is lower triangular with $L_{ii}=1$.
The rows of $U_{12}$ are elements of $\cup A_i$, and the rows of $U_{22}$ are
of the form  $a - a'$, $a$, $a' \in A_i$ for the same $i$.

Thus,
\begin{eqnarray*}
\rank{\Aact} &=& 
\rank{
\left[
\begin{matrix}
\mathrm{e}_{i_1} & \mathbf a_1 \\
\vdots & \vdots \\
\mathrm{e}_{i_{d}} & \mathbf a_d \\
\end{matrix} 
\right]}
+ n - d
\\
&=&
\rank{
\left[
\begin{matrix}
\mathrm{e}_{i_1} & \mathbf a_1 & b(i_1, \mathbf a_1) \\
\vdots & \vdots \\
\mathrm{e}_{i_{d}} & \mathbf a_d & b(i_d, \mathbf a_d)\\
\end{matrix} 
\right]}
+ n - d.
\end{eqnarray*}

Should the matrix $\Aact$ be singular, 
there will be $d$ linearly dependent vectors
of the form $[\mathrm e_i, \mathbf a, b(i,\mathbf a] \in \mathbb R^{n+s+1}$
contradicting the genericity of the lifting $b$.

The same argument shows that should there be another active constraint
$(i, \mathbf a)$, there would be $d+1 \le n+s_1$ 
linearly dependent vectors of the form $[\mathrm e_i, \mathbf a, b(i,\mathbf a] \in \mathbb R^{n+s+1}$, contradiction again.
\end{proof}

The flag $F=F_0 \subset \cdots \subset F_{n}$ will be assumed to 
be an affine flag `at infinity'.
This means that the equations for $F_d$ are of the form 
\begin{eqnarray*}
Q_{d+1} \mathbf \xi &=& R^{n-d} r_{d+1} \\
&\vdots& \\
Q_{n} \mathbf \xi &=& R r_n
\end{eqnarray*}
with generic $Q_i$, non-zero $r_i$ and $R \rightarrow +\infty$.
\glossary{$Q_i$ & Unit vector orthogonal to $F_{i-1}$ in $F_i$.}
\glossary{$R$ & non-standard number, $R>k$ for all $k \in \mathbb R$.}
The values of $\mathbf \xi$ and $\mathbf \lambda(\mathbf \xi)$
are now polynomials in $R$.
\glossary{$B$ & Inverse to the matrix $\Aact$ of active constraints.}
Let $B = \Aact^{-1}$. The ansatz below is the key to represent 
those polynomials by their constant term:
\begin{equation}\label{ansatz}
\left[
\begin{matrix}
\lambda \\
\xi
\end{matrix}
\right]
=
\left[
\begin{matrix}
\lambda_0 \\
\xi_0
\end{matrix}
\right]
+
B
\left[ \begin{matrix} 0 \\ \vdots \\ 0 \\ R^{n-d} r_{d+1} \\ \vdots \\ R r_n\end{matrix} \right]
\hspace{1em}
\text{with}
\hspace{1em}
\Aact
\left[
\begin{matrix}
\lambda_0 \\
\xi_0
\end{matrix}
\right]
=
\bact = 
\left[
\begin{matrix}
b(i_1, a_1) \\
\vdots\\
b(i_d, a_d) \\
0\\
\vdots \\
0
\end{matrix}
\right]
\end{equation}

Only $\lambda_0$ and $\xi_0$ need to be stored in memory, the
rest of the polynomial is implicit once we know $\Aact$ and $B$.

Inactive constraints $(i,a)$ (for $R \rightarrow \infty$) satisfy
\begin{equation}\label{inactive}
\begin{split}
R^{n-d} r_{d+1} \left[ -\mathrm{e}_i, a \right] B \mathrm e_{s+d+1}
+
&\cdots
+
R r_{n} \left[ -\mathrm{e}_i, a \right] B \mathrm e_{n+s} + \\
&+
\left[ -\mathrm{e}_i, a \right] B
\left[
\begin{matrix}
\lambda_0 \\
\xi_0
\end{matrix}
\right]
- b(i,a)
<
0
\end{split}
\end{equation}

The expression above may be interpreted as a polynomial in $R$ or, as $R \rightarrow \infty$ as a
non-standard number. It is strictly negative if and only if the higher-order non-zero coefficient is strictly negative. The reader can check
that only the sign of the $r_j$'s matters:
\begin{lemma}
Assume that $r_j > 0$ for all $j$. If \eqref{inactive} is strictly negative for $R$ large enough,
then for any other choice of the $r_j > 0$ it will remain negative for $R$ large enough.
\end{lemma}

Now we consider the effect of `dropping' the $j$-th constraint. 
As before, we set 
\[
\left[
\begin{matrix}
\Delta_j \lambda \\
\Delta_j \xi
\end{matrix}
\right]
=
-B \mathrm{e}_j
.
\]

By Lemma~\ref{lem:balancing},
the corresponding edge is of the form
\[
\{
t 
\Delta_j \xi, \ 
t \in (0, I_j)
\}
\]
where $I_j$ can be determined as in Lemma~\ref{lem:Ij}. 
We will need the polynomial
\[
t(i,a) (R) = \sum_{l=0}^{n-d} t_l(i,a) R^{l}
\glossary{$t_{i,a}(R)$ & Scores for inactive constraints.}
\]
where
\[
t_0(i,\mathbf a) = \frac{ a (\xi_0)_i - b(i,a) - (\lambda_0)_i}
{ [-\mathrm e_i,\mathbf a]  B \mathrm e_j }
\]
and for $l=1, \dots, n-d$,
\[
t_l(i,\mathbf a) = \frac{ [-\mathrm e_i,\mathbf a]  B \mathrm e_{n+s+1-l}} 
{ [-\mathrm e_i,\mathbf a]  B \mathrm e_j }
r_{n+1-l} 
.
\]

Let $\mathcal C$ be the set of inactive constraints with $t(i,a)(R) > 0$ once $R$ is large enough.
There may be inactive constraints with $[ -\mathrm e_i, \mathbf a]  B \mathrm e_j = 0$ but 
those are not eligible as elements of $C$. Assuming $C$ not empty,
$\argmin_C t(i,a)(R)$ denotes the constraint $(i,a) \in C$ that is minimal, once $R$ is large enough.
\medskip
\par

\begin{lemma}\label{lem:Ij1}
Let $\xi \in \revc{G}_d$ be a vertex. Assume all the notations above. If $C$ is not empty,
Let $(i^*, a^*) = \argmin_{(i,a) \in C} t(i,a)(R)$, $t^*=t(i^*, a^*)=\sum_{l=0}^{n-d}t_l^* R^l$.  
Then, 
\begin{enumerate}[(a)]
\item \label{Ij1a} $(i^*,a^*)$ is uniquely defined.
\item \label{Ij1b} If $R$ is large enough, $[\xi, \xi']$ is a segment of $\revc{G}_d$, where 
$\xi'= \xi + t^* \Delta \xi$. 
\item \label{Ij1c} Let $\lambda'= \lambda + t^* \Delta \lambda$. 
Let $B'$ be the inverse to the matrix of active constraints at $\mathbf \xi'$. 
Then
\[
\left[
\begin{matrix}
\lambda' \\
\xi'
\end{matrix}
\right]
=
\left[
\begin{matrix}
\lambda_0' \\
\xi_0'
\end{matrix}
\right]
+
B'
\left[ \begin{matrix} 0 \\ \vdots \\ 0 \\ R^{n-d} r_{d+1} \\ \vdots \\ R r_n\end{matrix} \right]
\]
with
\[
\left[ \begin{matrix} \lambda_0' \\ \xi_0' \end{matrix} \right]
=
\left[ \begin{matrix} \lambda_0 \\ \xi_0 \end{matrix} \right]
-
t_0^* B \mathrm{e_j}
.
\]
\end{enumerate}
\end{lemma}

\begin{proof}
We prove item (\ref{Ij1b}) before unicity. Let $0 < \tau < t^*$.
For active constraints $(k, a_k)$ except 
the $j$-th one (that we `dropped'),
$
-\tau [-e_k , a_k] B e_j = -\tau \delta_{jk}=0$ so that
\[
a_k (\xi + \tau \Delta_j \xi ) - b(k, a) = \lambda_k + \tau \Delta_j \lambda_k
.
\]
We claim that all the other constraints $(k, a')$ satisfy
\[
a' (\xi + \tau \Delta_j \xi ) - b(k, a') < \lambda_k + \tau \Delta_j \lambda_k
.
\]

For the $j$-th constraint, this follows from
$-\tau [-e_j , a_j] B e_j = -\tau < 0$. Similarly, if $(k, a')$ is an inactive
constraint not in $C$,
either $-\tau [-e_k , a'] B e_j = 0$ or $[-e_k , a'] B e_j$ and $a_k \xi - b(k, a^*) - \lambda_k$
have different sign. Since the latter is negative, $[-e_k , a'] B e_j > 0$ and
$-\tau [-e_k , a'] B e_j \le 0$.

Now assume $(k, a') \in C$: since the numerators of $t(i^*,a^*) \le t(k, a_k)$ are all negative,
we have the inequality $[e_{i^*}, a^*] B e_j \ge \tau[-e_k, a_k] B e_k$.

\begin{equation*}
\begin{split}
a_k (\xi + \tau \Delta_j \xi ) - b(k, a) - \lambda_k + \tau \Delta_j \lambda_k
&=
a_k \xi - b(k, a) - \lambda_k -\tau [e_k, a] B e_j\\
&=
(t(k,a) - \tau)  [e_k, a] B e_j\\
&\le
(t(k,a) - \tau)  [e_{i^*}, a^*] B e_j\\
&\le 
0\\
\end{split}
\end{equation*}
with equality if $(k,a)=(i^*,a^*)$.
\medskip
\par
We can prove item~\eqref{Ij1a} now. 
Should unicity fail, there will be $n+d+1$ active constraints for $\xi + t^* \Delta_j \xi$,
which contradicts Lemma~\ref{lem:genericity}. Therefore, the minimum of $w$ is attained in a unique
$(i^*, a^*) \in C$.

\medskip
\par
Let $\Aact'$ be the matrix of active constraints for $\xi'$ and let $B'$ be its inverse.
In order to prove item~\eqref{Ij1c}, we will first check that for $1 \le l < n-d$, 
\[
B' \mathrm e_{n+s+1-l} r_{n-l+1}
=
B \mathrm e_{n+s+1-l} r_{n-l+1}
-
t_l^* B \mathrm{e_j}.
\]
Notice that $\Aact' = \Aact - \mathrm e_j v$, with $v=[-\mathrm e_j, a_j] - [-\mathrm e_{i^*}, a^*]$. Also, let
$\bact'= \bact + \mathrm e_j (b(i^*,a^*)-b(j,a_j))$.
By the previous item and by construction,
\[
\Aact' 
\left[ \begin{matrix}  \lambda + t \Delta_j \lambda
\\ \xi +t \Delta_j \xi \end{matrix} \right]
=
\bact'
\]

The invertibility of $\Aact'$ follows from Lemma~\ref{lem:genericity}.
Now we apply Lemma~\ref{rank_one_updates} to obtain an expression for $B'$:
\[
B'= B + \frac{1}{1- v^T B e_j} B e_j v^T B
.
\] 
Note that $v^T B = e_j^T - [-\mathrm e_{i^*} a^* ] B$. Thus,
for each $1 \le l \le n-d$,
\begin{eqnarray*}
r_{n+s-l} B' e_{n+s+1-l} &=&
r_{n+s-l} B e_{n+s+1-l}\\
&& + r_{n+s-l} \frac{[-\mathrm e_j, a_j]B e_{n+s+1-l} - [-\mathrm e_{i^*}, a^*]B e_{n+s+1-l}}{[-\mathrm e_{i^*} a^* ] Be_j} B e_j 
\\
&=&
r_{n+s-l} B e_{n+s+1-l} 
- t_{l}^* B e_j 
\end{eqnarray*}

Hence,
\[
B'
\left[ \begin{matrix} 0 \\ \vdots \\ 0 \\ R^{n-d} r_{d+1} \\ \vdots \\ R r_n\end{matrix} \right]
=
B
\left[ \begin{matrix} 0 \\ \vdots \\ 0 \\ R^{n-d} r_{d+1} \\ \vdots \\ R r_n\end{matrix} \right]
- \sum_{l=1}^{n-d} R^{l} t_{l} B e_j 
\]
Thus,
\begin{eqnarray*}
\left[ \begin{matrix} \lambda' \\ \xi' \end{matrix} \right]
&=&
\left[ \begin{matrix} \lambda \\ \xi \end{matrix} \right]
-
t^* B e_j
\\
&=&
\left[ \begin{matrix} \lambda_0 \\ \xi_0 \end{matrix} \right]
+
B
\left[ \begin{matrix} 0 \\ \vdots \\ 0 \\ R^{n-d} r_{d+1} \\ \vdots \\ R r_n\end{matrix} \right]
-
t^* B e_j
\\
&=&
\left[ \begin{matrix} \lambda_0 \\ \xi_0 \end{matrix} \right]
-
t_0^* B e_j
\end{eqnarray*}

\end{proof}

Lemma \ref{lem:Ij1} allows us to explore each of the sets $\revc{G}_d$, $1 \le d \le n$ and hence to
produce $X_d$, as long as we have at least one vertex from each connected component of $\revc{G}_d$.
Those vertices can be recovered from $X_{d-1}$ by using 
Theorem~\ref{th:connected} and the lemma below.

The following Lemma allows to find starting points in $\revc{G}_{d+1}$ by exploring $X_{d}$.
In order to do that, we `drop' the $s+d+1$-th constraint. Let $w$ and $C$ be defined
accordingly:
\[
t(i,a) (R) = \sum_{l=0}^{n-d} t_l(i,a) R^{l}
\]
where
\[
t_0(i,\mathbf a) = \frac{ a (\xi_0)_i - b(i,a) - (\lambda_0)_i}
{ [-\mathrm e_i,\mathbf a]  B \mathrm e_{s+d+1} }
\]
and for $l=1, \dots, n-d$,
\[
t_l(i,\mathbf a) = \frac{ [-\mathrm e_i,\mathbf a]  B \mathrm e_{n+s+1-l}} 
{ [-\mathrm e_i,\mathbf a]  B \mathrm e_{s+d+1} }
r_{n+1-l} 
.
\]

Let $\mathcal C$ be the set of inactive constraints with $t(i,a)(R) \ne 0$ once $R$ is large enough.
It is important to notice that for $R$
large, item (d) of Theorem~\ref{th:connected} reads:
\begin{enumerate}
\item[(d')] All points in $\revc{G}_d \cap F_{d-1}$ are in the interior of a half line of $\revc{G}_d$.
\end{enumerate}

Therefore, all constraints in $\mathcal C$ have value of the same sign (positive or negative).

\begin{lemma}\label{lem:Ij2}
With the notations of Lemma~\ref{lem:Ij1}, let $j=s+d+1$.
Let $\xi \in X_{d}$ be a vertex. If $C$ is not empty,
Let $(i^*, a^*) = \argmin_{(i,a) \in C} |t(i,a)(R)|$, $t^*=t(i^*, a^*)=\sum_{l=0}^{n-d}t_l^* R^l$.  
Then, 
\begin{enumerate}[(a)]
\item \label{Ij2a} $(i^*,a^*)$ is uniquely defined.
\item \label{Ij2b} If $R$ is large enough, $[\xi, \xi')$ is a half line of $\revc{G}_d$, where 
$\xi'= \xi + t^* \Delta \xi$. Moreover, $\xi'$ is a point of $\revc{G}_{d+1}$.
\item \label{Ij2c} Let $\lambda'= \lambda + t^* \Delta \lambda$.
Let $B'$ be the inverse to the matrix of active constraints at $\mathbf \xi'$. 
Then
\[
\left[
\begin{matrix}
\lambda' \\
\xi'
\end{matrix}
\right]
=
\left[
\begin{matrix}
\lambda_0' \\
\xi_0'
\end{matrix}
\right]
+
B'
\left[ \begin{matrix} 0 \\ \vdots \\ 0 \\ R^{n-d-1} r_{d+2} \\ \vdots \\ R r_n\end{matrix} \right]
\]
with
\[
\left[ \begin{matrix} \lambda_0' \\ \xi_0' \end{matrix} \right]
=
\left[ \begin{matrix} \lambda_0 \\ \xi_0 \end{matrix} \right]
-
t_0^* B \mathrm{e_j}
.
\]
\end{enumerate}
\end{lemma}

\begin{proof}
The proof of items \eqref{Ij2a} and \eqref{Ij2b} is similar to the proof of Lemma~\eqref{lem:Ij1}. Therefore
we will only prove item \eqref{Ij2c}.

Let $\Aact'$ be the matrix of active constraints for $\xi'$ and let $B'$ be its inverse.
As before, $\Aact' = \Aact - \mathrm e_j v$ where $v=[0, Q_{d+1}] - [-\mathrm e_{i^*}, a^*]$ is the difference between the $j$-th row of
$\Aact$ and the $j$-th row of $\Aact'$, and
$\bact'= \mathrm e_{d+1} b(i^*,a^*)$. We have
\[
\Aact' 
\left[ \begin{matrix}  \lambda + t \Delta_j \lambda
\\ \xi +t \Delta_j \xi \end{matrix} \right]
=
\bact'
\]

The invertibility of $\Aact'$ follows from Lemma~\ref{lem:genericity}, and
\[
B'= B + \frac{1}{1- v^T B e_{s+d+1}} B e_{s+d+1} v^T B
\] 
with $v^T B = e_{s+d+1}^T - [-\mathrm e_{i^*} a^* ] B$.
As in Lemma~\ref{lem:Ij1} item~\eqref{Ij1c}, for $1 \le l < n-d$, 
\[
B' \mathrm e_{n+s-l+1} r_{n-l+1}
=
B \mathrm e_{n+s-l+1} r_{n-l+1}
-
t_l^* B \mathrm e_{s+d+1}.
\]

When $l=n-d$, $t_{l}^* = r_{d+1}$. 
Therefore $R^{n-d} r_{d+1} B e_{s+d+1} - R^{n-d} t_{n-d} B e_{s+d+1}=0$ and
\begin{eqnarray*}
\left[ \begin{matrix} \lambda' \\ \xi' \end{matrix} \right]
&=&
\left[ \begin{matrix} \lambda \\ \xi \end{matrix} \right]
-
t^* B e_j
\\
&=&
\left[ \begin{matrix} \lambda_0 \\ \xi_0 \end{matrix} \right]
+
B
\left[ \begin{matrix} 0 \\ \vdots \\ 0 \\ R^{n-d} r_{d+1} \\ \vdots \\ R r_n\end{matrix} \right]
-
t^* B e_j
\\
&=&
\left[ \begin{matrix} \lambda_0 \\ \xi_0 \end{matrix} \right]
-
t_0^* B e_j
\end{eqnarray*}

\end{proof}

There is a reciprocal to Lemma~\ref{lem:Ij2}.

\begin{lemma}\label{lem:Ij3}
Let $\xi'$ be a vertex from $\revc{G}_{d+1}$.
Let $q$ be the unique integer such that $m_q(\xi') > \mm{q}{d}$. 
Then there is $\xi \in X_{d}$ such that $[\xi, \xi')$ is a half line
of $\revc{G}_d$ if and only if the following conditions hold for some active constraint
$(q,a)$ of $\xi'$ (say the $j$-th):
\begin{enumerate}[(a)]
\item The set $\mathcal C$ corresponding to `dropping' constraint $j$ is empty.
\item $[0\, Q_{d+1}]B'e_j < 0$.
\end{enumerate}
In that case, with the notations of Lemma~\ref{lem:Ij2},
\[
\left[
\begin{matrix}
\lambda\\
\xi
\end{matrix}
\right]
=
\left[
\begin{matrix}
\lambda_0\\
\xi_0
\end{matrix}
\right]
+
B
\left[ \begin{matrix} 0 \\ \vdots \\ 0 \\ R^{n-d} r_{d+1} \\ \vdots \\ R r_n\end{matrix} \right]
\]
with
\[
\left[
\begin{matrix}
\lambda_0\\
\xi_0
\end{matrix}
\right]
=
\left[
\begin{matrix}
\lambda'_0\\
\xi'_0
\end{matrix}
\right]
+
\frac{r_{d+1}}{[0, Q_{d+1}B'e_j]}
B'
e_j
.
\]
\end{lemma}

\section{The main algorithm}\label{sec:algorithm}

Here is a simplified version of the algorithm. Let $\mathcal V_d$ denote
the set of \rev{points (vertices)} of \rev{the tropical curve} $\revc{G}_d$ and 
$\mathcal E_d$ the set of finite segments \rev{(edges)}
in \revd{$\revc{G}_d$, $1 \le d \le n$.} The degree of the graph $\Gamma_d = (\mathcal V_d,
\mathcal E_d)$ is at most $\max \mm{i}{d-1}$. 
Consider now the graph
$\Gamma = (\mathcal V, \mathcal E)$
\glossary{$\Gamma =  (\mathcal V, \mathcal E)$& Graph to be explored. Union
of tropical curves.}
where $\mathcal V = \cup \mathcal V_d$ and $\mathcal E$ is the union
of all the $\mathcal E_d$ with the set of 
connecting segments $[\xi, \xi']$ from Lemma~\ref{lem:Ij2}, 
$\xi \in V_{d}$, $\xi'\in V_{d+1}$. The degree
of $\Gamma$ is at most $1+2\max \mm{i}{d-1}$. \rev{The AllMixedCells algorithm
is a graph walk through $\Gamma$}.

The algorithm will operate on two sets: a set $\Vexplore \subset \mathcal V$ to
explore, and a set $\Vknown \subset \mathcal V$ of already explored lower \revc{faces}. 
Notice that $\# \mathcal V \le \sum_{d=2}^n v_d$
where $v_d$ is the number of vertices of $\revc{G}_d$. Therefore
\glossary{$v_d$ & number of vertices of $\revc{G}_d$.}
$\#\mathcal E \le (\sum_{d=2}^n v_d) \deg (\Gamma)/2$.

The operator $\mathrm{Visited}(\Vknown, L)$ returns true if $L \in \mathcal \Vknown$.
Otherwise, it inserts $L$ in $\Vknown$ and returns false. 
The simplified graph exploration algorithm is:
\medskip
\par

\centerline{\boxed{
\begin{minipage}{\textwidth - 2em}
\begin{trivlist}
\item {\bf Algorithm} {\sc AllMixedCells} \label{allmixedcells}
\item $\Vexplore \leftarrow \{ L_{F_0} \}$.
\item $\Vknown \leftarrow \emptyset$
\item {\bf While }$\Vexplore \ne \emptyset$,
\subitem Remove some $L_{\xi}$ from $\Vexplore$.
\subitem {\bf If not} {\sc Visited}$(\Vknown, L_{\xi})$
\subsubitem {\bf If} $L_{\xi}$ is a mixed cell, {\bf then} output $L_{\xi}$.
\subsubitem {\bf For} each neighbor $L_{\xi'}$ of $L_{\xi}$ in $\Gamma$, 
      insert $L_{\xi'}$ in $\Vexplore$.
\subitem Discard $L_{\xi}$.
\end{trivlist}
\end{minipage}
}}
\medskip
\par
We assume that elements of $\sqcup A_i$ are \revb{uniquely} labeled by an integer
from $1$ to $\sum \#A_i$. 
A \revc{face} $L$ in $\Vexplore$ is represented by its list of active constraints,
and each active constraint is represented by an integer. Thus, a \revc{face}
$L_{\xi}$ for $\xi \in \mathcal V_d$ is represented by a strictly increasing
list of $s+d$ integers. From this representation one can find the 
inverse $B$ to the matrix of active constraints in time
$O(n^{\omega})$ with \rev{$\omega<2.3728639$} 
the exponent of matrix multiplication \rev{\cite{VassilevskaWilliams,LeGall}}.
Once we obtained $B$ it is easy to compute $\mathbf \xi_0$ and $\mathbf \lambda_0$. 

The neighbors $L_{\xi'}$ can be listed by the formulas in
Lemmas~\ref{lem:Ij1} to ~\ref{lem:Ij3}. Not all inactive constraints need to be
tested. Suppose we drop the $j$-th constraint. All other active constraints
will be called the {\em remaining} constraints.

\begin{lemma}
Let $(i, \mathbf a)$, $(i,\mathbf a')$ be active constraints in 
the same lower \revc{face}. Then,
\[
[(\mathbf a, b(i,\mathbf a)), (\mathbf a', b(i,\mathbf a') )]
\]
is a sharp lower edge of 
\[
{\hat {\mathcal A}}_i = \conv{ \{ (\mathbf a, b(i,a)): a \in A_i \}} \subset \mathbb R^n \times \mathbb R.
\]
\end{lemma}

It follows from the Lemma that the only inactive constraints to
check are those $(i^*, a^*)$ so that 
$[(\mathbf a^*, b(i^*,\mathbf a^*)), (\mathbf a', b(i^*,\mathbf a') )]$
is a sharp lower edge of ${\hat {\mathcal A}}_{i^*}$ for all remaining
constraints
$\mathbf a'$. Thus,

\begin{lemma} \label{cost:nbhd}
Assume that a lower \revc{face} $L_{\mathbf \xi}$ is
given and that the matrix $B$ inverse to the matrix of active constraints
is known.
\begin{enumerate} 
\item If $s = 1$ and $d=n$, the neighboring lower \revc{face} of 
$L_{\mathbf \xi}$ can be produced in time at most
\[
O( n^2 \# A_1 ).
\]
\item If $n=s$ and $m_1 = \cdots = m_n = 2$ and the sharp lower edges of
the ${\hat {\mathcal A}}_i$'s 
are known, then the neighboring lower \revc{faces} of 
$L_{\mathbf \xi}$ can be produced in time at most
\[
O( n^2 \sum ( E_i )).
\]
where $E_i$ is the maximal number of lower sharp edges incident to
a sharp vertex of ${\hat {\mathcal A}}_i$.
\glossary{$E_i$ & Degree of 1-skeleton of lifting of $A_i$}
\end{enumerate}
\end{lemma}

\begin{proof}
If $s=1$, there are $n+1$ active constraints that may be 'dropped'
by using the formulas in Lemmas~\ref{lem:Ij1} to ~\ref{lem:Ij3}.
Since $d=n$, testing an inactive constraint costs $O(n)$ arithmetic
operations. Testing all inactive constraints, the $\argmin$ in the formulas
costs at most $n^2 \# A_1$.

In the second case, there are two or three constraints that may be
'dropped'. Testing an inactive constraint costs $O(n(n+1-d)) \le O(n^2)$
arithmetic operations. There are at most $\sum E_i$ inactive constraints
to test for the $\argmin$.

\end{proof}
\begin{remark} 
One does not need to explore edges going from $\mathcal V_d$ to
$\mathcal V_{d-1}$. 
\end{remark}

\begin{remark} If $L_{\mathbf \xi}$ and $L_{\mathbf \xi'}$ are
adjacent and $B$ is known, the inverse $B'$ of the matrix
of active constraints of $L_{\xi'}$ can be produced by a rank-one update
with cost $O(n^2)$. 
\end{remark}

In order to bound the cost of each {\sc While} iteration inside
algorithm {\sc AllMixedCells}, we still need to bound the cost
of storing and retrieving data from sets $\Vexplore$ and $\Vknown$.
We represented each lower \revc{face} $L$ by a list of active constraints, that
is a unique strictly increasing list of integers $(l_1, \cdots, l_{s+d})$
from $1$ to $\sum \# A_i$. To each of those lists, we associate a 
{\em hash value},
that is a real number in $[0,1]$ computed by the formula:
\[
\sigma(L) = H_1 l_1 + \cdots + H_{s+d} l_{s+d} \mod 1
.
\]
where the $H_j \in [0,1]$ are uniformly distributed 
random numbers computed once and for
all before starting {\sc AllMixedCells}. The probability of a
{\em collision}, that is of two different lower \revc{faces} having the
same value of $\sigma$, is zero. The {\em hash function} $\sigma$
will be used to order the lower \revc{faces} in such a way that the
comparison cost is one.

The sets $\Vexplore$ and $\Vknown$ are represented by balanced
trees\cite{KnuthV3}*{Ch.6}. For lower \revc{faces} in $\Vknown$, only the hash 
value needs to be
stored. The cost of the operator {\sc Visited} is therefore bounded
above by $O(\log (\# \mathcal V))$. For the case of $\Vexplore$,
we can also store the list of active constraints. Thus:
\begin{lemma} Under the assumptions above, the cost of
storing or retrieving a value in $\Vexplore$ or $\Vknown$ is at most
$O(n+ \log (\# \mathcal V))$.
\end{lemma}

The algorithm below supports the bounds in Theorem~\ref{main}. \rev{
It proceeds in two steps. First it computes the sharp lower edges of
each $\hat {\mathcal A_i}$ (See Lemma \ref{cost:nbhd}). For each $\mathbf a \in A_i$, its lifting $(\mathbf a, b_i(\mathbf a))$ is a sharp lower vertex of $\hat {\mathcal A_i}$
if and only if it belongs to the border of a sharp lower edge. We will now
replace the lifting $b_i: A_i \rightarrow [0,1]$ with a new general lifting,
$\tilde b_i: A_i \rightarrow [0,2]$ with
}
\[
\rev{
\tilde b_i (\mathbf a) = \left\{ \begin{array}{ll}
b_i(\mathbf a) & \text{if $(\mathbf a, b_i(\mathbf a))$ a sharp lower vertex, and}\\
b_i(\mathbf a)+1 &\text{otherwise.}
\end{array}\right.
}
\]

\rev{The second step is the {\sc AllMixedCells} algorithm applied to the
lifting $(\tilde b_1, \dots, \tilde b_n)$}.

\centerline{\boxed{
\begin{minipage}{\textwidth - 2em}
\begin{trivlist}\label{allmixedcellsfull}
\item {\bf Algorithm} {\sc AllMixedCellsFull}$(n,A_1, \dots A_n)$ \label{support-alg}
\item {\bf Draw} a random lifting $b: \sqcup_i A_i \rightarrow [0,1]$, each coordinate
uniformly distributed.
\item {\bf Draw} uniform random numbers $H_1, \dots, H_{2n}$ uniformly distributed
in $[0,1]$.
\item {\bf Draw} an orthogonal matrix $Q$ uniformly distributed with respect to the Haar measure of $SO(n)$ and define $F_i$ as the space orthogonal to the last $n-i$ columns of $Q$.
\item {\bf For} $i=1, \dots, n$ invoke {\sc AllMixedCells} with $s=1$ and
input $A_i, b_i$. Produce a list of the lower edges of $\hat A_i$. 
\item {\bf Set} $V \leftarrow 0$.
\item
Invoke {\sc AllMixedCells} with $s=n$, $m_i=2$ and input $A_1, \dots, A_n$,
\rev{$\tilde b_1, \dots, \tilde b_n$}.
For each mixed lower \revc{face}, add the volume of the mixed cell to $V$
and output the lower \revc{face}.
\item {\bf Output} $V$.
\end{trivlist}
\end{minipage}
}}

\medskip
\par

We can prove the first part of Theorem~\ref{main}:

\begin{lemma}
Let $d_i = \dim \conv{A_i}$ and $V_i = d_i! \vol {\conv{A_i}}_{d_i} 
$.
Let 
\[
T = (\sum_{i=2}^d v_i) \left( n^2 \sum_{i=1}^n E_i + \log \sum_{i=2}^n v_i \right),
T' = (\max V_i) \left( n^2 \sum_{i=1}^n \# A_i + \log \max_{i=1, \dots, n} V_i \right).
\]
The algorithm {\sc AllMixedCellsFull} performs at most $O(T+T')$ arithmetic operations.
\end{lemma}

\begin{proof}
Uniform random numbers $b(i,a)$ and $H_i$ can be obtained at unit cost.
To produce the matrix $Q \in SO(n)$, one first produces a random matrix $A$,
where each coordinate $A_{ij}$ is an independently distributed 
$N(0,1)$ random numbers. By using standard algorithms like the Box-Müller
method, one obtains each $A_{ij}$ within bounded cost. Then, $Q$ may be
produced by computing the QR factorization of $A$, or equivalently be
Gram-Schmidt orthonormalization of the columns of $A$.

Computing the dimension $d_i$ of each $A_i$ is easy.
There are at most 
\[
d_i! \vol {\conv{A_i}} _{d_i} 
\] 
lower \revc{faces} to
explore, and the first one can be found by standard linear programming
techniques. The cost of each exploration step is bounded by Lemma~\ref{cost:nbhd}(1). Therefore, the total cost is $O(T')$.

The last call to {\sc AllMixedCells} must explore
$\# \mathcal V = \sum v_i$ lower \revc{faces}, and the cost of exploring
each one was bounded in Lemma~\ref{cost:nbhd}(2). The total cost of
this step is therefore $O(T)$.
\end{proof}

\revb{\begin{remark}The Box-Müller method requires the use of a logarithm,
a square root, sine and cosine. The Gram-Schmidt method or QR factorization
uses square roots. Rather than approximating those functions, I 
assumed that they are available at unique cost. 
\end{remark}}

\section{Deterministic complexity analysis}
\label{sec:deterministic}

Let $v_d$ be the number of vertices of $\revc{G}_d$. The number $v_d$ depends on the lifting $b: A_1 \sqcup \cdots
\sqcup A_s \rightarrow \mathbb R$ that is assumed
generic, but fixed.
It also depends on the flag $(F_0 \subset \cdots \subset F_n)$ that we take of the form 
\[
F_{n-d}(R) = 
\{
\xi \in \mathbb R^n: Q_{d+1} \xi = R^{n-d} , \dots, Q_{n-1} \xi = R^2,  \revb{Q}_n \xi = R
\}
\]
where $Q_1, \dots, Q_n$ are the rows of a generic matrix in $O(n)$, and $R\revb{>0}$ is assumed to be large enough.
\medskip
\par
Let $A_0 = \{ 0, \mathrm e_1, -\mathrm e_1, \mathrm e_2, \dots, -\mathrm e_n\}$ and
$\mathcal A_0 = \conv{A_0}$. The solid $\mathcal A_0$ is also
known as the {\em $n$-orthoplex} and denoted by $\beta_n$. Alternative
terminologies are {\em $n$-cross} and {\em cocube}. Notice that
\[
\frac{1}{\sqrt n} B^n \subset \beta_n \subset B^n.
\]
The main result in this section is:

\begin{proposition}\label{prop:deterministic}
\[
v_d \le n! V( \mathcal A_1, \mm{1}{d-1} ; \cdots ; \mathcal A_s, \mm{s}{d-1}; \mathcal A, 1; \beta_n, n-d )
.
\]
\end{proposition}
Towards the proof of Proposition~\ref{prop:deterministic} we extend the lifting $b$ to $A_0=\beta_n$
by
\[
b(0, 0) = 0
\hspace{2em}
\text{and}
\hspace{2em}
b(0, \pm \mathrm e_i) = S
\]
where $S>0$ will be determined later.

\begin{definition} Let $W \subset A_0 \sqcup \cdots \sqcup A_s$. The {\em set of viable points} for $W$
is
\[
\begin{split}
\Xi(W) = \left\{ \xi \in \mathbb R^n: 
\forall (i,a)\right. & \in A_0 \sqcup \cdots \sqcup A_s, 
\\
&
a \xi -\lambda_i(\xi) - b_i(a) \le 0 
\\
&
\left.
\text{with equality iff $(i,a) \in W$}
\right\}.
\end{split}
\]
\end{definition}
To every $W \subset A_0 \sqcup \cdots \sqcup A_s$ we associate the pair
\begin{equation}\label{active}
\Aact  = \Aact(W) = 
\left[ 
\begin{matrix}
\vdots & \vdots \\
-\mathrm e_i & a \\
\vdots & \vdots \\
0 & Q_{e+1} \\
\vdots & \vdots \\
0 & Q_n
\end{matrix}
\right]
\text{ and }
\bact = \bact(W) =
\left[ 
\begin{matrix}
\vdots \\
b_i(a)\\
\vdots \\
0 \\
\vdots \\
0 
\end{matrix}
\right]
\end{equation}
where $e = \# W - (s+1)$ and $(i,a)\in W$.
We say that $W$ is linearly independent if $\Aact(W)$ is invertible. 
In that case we also set $B=B(W) = \Aact(W)^{-1}$.

\begin{proof}[Proof of Proposition~\ref{prop:deterministic}]
We start by fixing $S$. To that end, we notice that for every linearly independent set $W$
such that 
\[
\exists S,\, R(S)>0\, \text{ s.t. }\  \forall R > R(S), \
\Xi(W) \cap F_e(R) \ne \emptyset,
\]
equation \eqref{ansatz} determines $\xi(R) \in  \Xi(W) \cap F_e(R)$ as
\[
\left[
\begin{matrix}
\lambda(R) \\
\xi(R)
\end{matrix}
\right]
=
B \,
\bact
+
B
\left[
\begin{matrix}
0 \\
\vdots \\
0 \\
R^{n-e} \\
\vdots\\
R
\end{matrix}
\right]
.
\]
The first term is a constant when $W \cap A_0 = \{0\}$, otherwise
it is \revb{affine in $S$}. The second term is a polynomial in $R$.

For any $(i,a) \not \in W$, write
\[
t(i,a)=
\left[
\mathrm e_i \ a \right] 
\left[
\begin{matrix}
\lambda(R) \\
\xi(R)
\end{matrix}
\right]
-
b_i(a)
=
f_{W,i,a}(R) + g_{W,i,a} S
\]
where $f_{W,i,a}$ is a polynomial in $R$ and $g_{W,i,a}$ is a constant,
vanishing when $W \cap A_0 = \{0\}$. By hypothesis, the expression above
is negative for $R$ large enough. This means that the leading term
of $f_{W,i,a}$ is strictly negative.

As there are finitely many $W,i$ and $a$, there exists a uniform constant $\revb{R_d} > 0$
such that when $R\ge \revb{R_d}$, \revb{each of the} $f_{W,i,a}(R)$ \revb{for all $d \le e \le n-1$} is strictly negative and non-increasing. This constant \revb{$R_d$} is independent of $S$.
Now pick $S$ large enough so that if $W \cap A_0 = \{0\}$, $\| \xi(\revb{R_d}) \|_{\infty} < S$.
Notice that if $(0,0) \in W$ and $(0,\pm \mathrm e_i) \in W$ then forcefully $\xi_i = \pm S$.

\medskip
\par
For each $d \le e \le n$, let $\mathcal W_e$ be the class of all subsets $W \subset A_0 \sqcup \cdots
\sqcup A_s$ such that
\begin{enumerate}[(a)]
\item $(0,0) \in W$.
\item $W$ is linearly independent.
\item $\# W = s+e+1$.
\item $\#W \cap A_0 = 1 + e - d$.
\item $\#W \cap A_i \ge 1 + \mm{i}{d-1}$, \text{and}
\item For $R$ large enough, $\Xi(W) \cap F_e(R) \ne \emptyset$. This \revb{holds} in
particular for some $R$ with $\| \xi(R)\|_{\infty} \le S$. 
\end{enumerate}

{\bf Induction hypothesis in $e \in \{d, d+1, \dots, n\}$:} For every vertex $\xi(R)$ of $\revc{G}_{d}$, there is one and
only one $W \in \mathcal W_e$ with $\xi(R) \in \Xi(W)$ for $R>\revb{R_e}$ \revb{for} some $\revb{R_e}$ with $\| \xi(\revb{R_e})\|_{\infty} \le S$.

{\bf Base step $e=d$:} Let $W = \{ (0,0) \} \cup \{ (i,a) \in A_i, i \ge 1: 
-\lambda_i(\xi(R)) + a \xi(R) - b(i,a) = 0 \}$. For this choice of $W$, $g_{W,i,a}=0$
at all inactive constraints. 
By construction of $\revc{G}_d$, $W \in \mathcal W_d$ and $\xi(R) \in \Xi(W)$.

{\bf Induction step:} Let $W \in \mathcal W_e$, $d<e<n$. Assume after reshuffling indexes and changing signs that
$W \cap A_0 = \{ 0, \mathrm e_1, \dots, \mathrm e_{e-d} \}$.
For some value of $\revb{R_{e+1} > R_e}$, the curve $(\xi_{e-d+1}(R), \dots, \xi_{n}(R))$ will exit the
hypercube $\max_{i>e-d} ( |\xi_i|) = S$. Say this happens for $\xi_{e-d+1}(\revb{R_{e+1}})$. Then we
set $W'= W \cup \{ (0, \mathrm e_{e-d+1}) \}$. 

The point $\xi(\revb{R_{e+1}})$ belongs to $\Xi(W')$. 
By construction, $W'$ is linearly independent. So we can construct $\Aact(W')$ and $B(W')$, 
and hence all the $f_{W',i,a}$ and $g_{W',i,a}$ for inactive constraints $(i,a) \not \in W'$.
With $R \ge R_{e+1}$, set:
\[
\left[
\begin{matrix}
\lambda'(\revb{R})
\\
\xi'(\revb{R}) 
\end{matrix}
\right]
=
B(W')\  \bact(W') 
+
B(W')
\left[
\begin{matrix}
0 \\
\vdots \\
0 \\
\revb{R}^{n-e-1} \\
\vdots\\
\revb{R}
\end{matrix}
\right].
\]
Then for all inactive constraints,
\[
-[\mathrm e_i ,   a] 
\left[
\begin{matrix}
\lambda'(\revb{R})
\\
\xi'(\revb{R}) 
\end{matrix}
\right]
-b(a)
=
f_{W',i,a}(\revb{R}) + g_{W',i,a}\revb{(S)}
.\]
Since $\revb{R} \ge \revb{R_{e+1} > R_e}$, $f_{W',i,a}(R') < 0$ is negative and non-increasing. 
Thus
\[
f_{W',i,a}(\revb{R}) + g_{W',i,a}\revb{(S)}
\le
f_{W',i,a}(\rev{R_{e+1}}) + g_{W',i,a}\revb{(S)} < 0
.
\]
It follows that for $\revb{R \ge R_{e+1}}$,
$\Xi(W') \cap F_e(R') \ne \emptyset$. Moreover,
$\| \xi'(\revb{R_{e+1}}) \|_{\infty} = S$.

{\bf Conclusion}. By induction, we can associate injectively
to each vertex of $\revc{G}_d$, an element of $\mathcal W_n$ which is a mixed
cell for one $(A_0, n-d; A_1, \mm{1}{d-1} + \delta_{1p};
\cdots; A_s, \mm{s}{d-1} + \delta_{sp})$. The volume of
such a mixed cell is an integral multiple of $1/n!$. 
Proposition~\ref{prop:deterministic} follows.
\end{proof}

\section{Average complexity analysis}\label{sec:average}

Proposition~\ref{prop:deterministic} holds for a generic
lifting $b$ and for a sufficiently generic flag 
$F_0 \subset \cdots \revb{\subset} F_n$.
Now we assume that the orthogonal group $O(n)$ is endowed
with the Haar probability measure. \revb{For each} $Q \in O(n)$, 
\revb{let $Q_d$ be the $d$-th row of $Q$. Let $R$ be an
arbitrarily large positive
real number.
As before, the flag of affine subspaces is:}
$F_n = \mathbb R^n$,
\[
F_{d-1} = \{ \xi \in F_{d}: Q_d \xi = R^{n-d+1}\}
.
\]

Then we set $\bar v_d = \avg{Q \in O(n)} v_d$. The following
result should be compared with Proposition~\ref{prop:deterministic}:

\begin{proposition}\label{prop:average}
Suppose that $V(\mathcal A_1, m_1; \cdots ; \mathcal A_s, m_s) \ne 0$.
Then,
\[
\bar v_d \le \frac{n!}{2^{n-d}} V\left( \mathcal A_1, \mm{1}{d-1} ; \cdots ; \mathcal A_s, \mm{s}{d-1}; \mathcal A, 1; B^n, n-d \right)
.
\]
\end{proposition}

We will use the following fact to establish Proposition~\ref{prop:average}:
\begin{lemma}\label{flat}
Suppose that
\revb{the topological closure $\overline{\Xi(W)}$ of $\Xi(W)$} contains a line. Then $V(\mathcal A_1, m_1; \cdots;$ $\mathcal A_s, m_s) = 0$.
\end{lemma}

\begin{proof}
Let the line be $\xi_0 + t \dot \xi$. Without loss of generality,
assume that $a \dot \xi=0$ for all active constraints $(i,a)\in W$.
For inactive constraints, 
\[
-\lambda_i + a \xi_0 + t a \dot \xi\ \revb{\le}\ b(i,a) \ \forall t
.
\]
This is only possible if $a \dot \xi = 0$. So $A_1, \dots, A_s
\subset \dot \xi^{\perp}$.
\end{proof}

\begin{proof}[Proof of Proposition \ref{prop:average}]
\revb{The case $d=n$ follows from Proposition~\ref{prop:deterministic} so we assume $d<n$.}
Let $\mu_1+\cdots+\mu_s = d$. We define
$\mathcal W (\mu_1, \dots, \mu_s)$ as the class
of subsets $W \in A_1 \sqcup \cdots \sqcup A_s$
with $\# W \cap A_i = \mu_i+1$ and such that
$\Xi(W)$ contains \revb{an affine cone} of dimension $n-d$.

Every vertex of $\revc{G}_d$ corresponds to such a 
subset for some choice of $\mu_i \ge \mm{i}{d-1}$.
Indeed, the curve $\xi(R)$ obtained by varying $R$
large enough cannot be contained in an hyperplane
of codimension $d+1$.
So what we need to do is to count the average number
of elements of $\mathcal W (\mu_1, \dots, \mu_s)$.

Let $\Aact(W,Q)$ be as in \eqref{active} for a particular
choice of $Q$. Then,
\[
\max_{Q \in O(n)} \det (\Aact(W,Q))
=
\max_{\revb{Q_{d+1}, \dots, Q_n} \in B^n} \det (\Aact(W,Q))
\ge 
1
.
\]

Also, let $\mathcal W(\mu_1, \dots, \mu_s, Q)$ be the
set of $W \in \mathcal W(\mu_1, \dots, \mu_s)$ such that
$\Xi(W) \cap F_d(R)$ is not empty, for $R$ sufficiently large. 
\revb{Then,}

{\small
\begin{eqnarray*}
\avg{Q \in O(n)}\hspace{-2em}
&&\# \mathcal W (\mu_1, \dots, \mu_s; Q) \le \\
&\le&
\sum_{W \in  \mathcal W (\mu_1, \dots, \mu_s)}
\mathrm{Prob}[\Xi(W) \revb{\cap F_{n-d}(R)} \ne \emptyset ]
\\
&\le&
\sum_{W \in  \mathcal W (\mu_1, \dots, \mu_s)}
\left(
\mathrm{Prob}[\Xi(W) \revb{\cap F_{n-d}(R)} \ne \emptyset ]
\max_{\revb{Q_{d+1}, \dots, Q_n} \in B^n} \det (\Aact(W,Q))
\right)
\\
&\le&
\left(\max_{\revb{W 
}} \mathrm{Prob}[\Xi(W) \revb{\cap F_{n-d}(R)} \ne \emptyset]
\right)
\left(\sum_{W 
}
\max_{\revb{Q_{d+1}, \dots, Q_n} \in B^n} \det \Aact(W,Q)
\right)
\end{eqnarray*}}
The second term is bounded above by
\[
n!V(\mathcal A_1, \mu_1; \cdots; \mathcal A_s, \mu_s; B^n,n-d)
.
\]

We claim now that
\[
\mathrm{Prob}[\Xi(W) \cup F_{n-d}\revb{(R)} \ne \emptyset ]
\le
\frac{1}{
2^{n-d}}
.
\]
Indeed, let $W \in \mathcal W (\mu_1, \dots, \mu_s)$ and
admit that there is $Q$ such that
\[
\xi(R) \in \Xi(W) \cap F_{n-d}(R) 
\]
exists for $R$ large enough. Let
\[
I_n \ne
\Sigma =
\left[
\begin{matrix}
I_d \\
& \pm 1 \\
& & \ddots \\
& & & \pm 1
\end{matrix}
\right]
\]
be a non-trivial sign change matrix, and let $Q'= \Sigma'Q$.
Let $F_{n-d}'(R)$ be the induced flag. Suppose that
$\xi'(R) \in \Xi(W) \cap F_{n-d}'(R)$.

Since $\Xi(W)$ is a \revb{convex, the line segment $[\xi(R),\xi'(R)]$
is contained in $\Xi(W)$. Making $R \rightarrow \infty$, we conclude that
$\overline{\Xi(W)}$ contains a straight line. Then 
Lemma~\ref{flat} implies that the mixed volume of the
$\mathcal A_i$ vanishes, contradiction.}
Therefore, 
$\Xi(W) \cap F_{n-d}'(R) = \emptyset$.

\revb{The multiplicative group of all the sign matrices $\Sigma$ as
above preserves the Haar metric in $O(n)$, and has order $2^{n-d}$.
Since only one of the $\Sigma Q$ can have $\Xi(W) \cap F_{n-d}(R)
\ne \emptyset$, we deduce that 
\[
\mathrm{Prob}[\Xi(W) \cup F_{n-d}\revb{(R)} \ne \emptyset ]
\le
\frac{1}{
2^{n-d}}
.
\]}
\end{proof}

\section{Implementation notes}
\label{sec:implementation}

In this section, I describe several techniques that were used
in the software and have an effect on the experiments. 
\paragraph{\bf Hermite normal form.} The polytopes $A_1, \dots, A_n$ are translated so that each one
has a vertex equal to zero. Then, the program computes a basis for the lattice
generated by $\cup(A_i)$ through a Hermite normal form factorization. 
The supports $A_i$ are then rewritten into lattice
basis coordinates. This may reduce the mixed volume in families of examples
such as {\em Cyclic-n} or {\em Gridanti-n}. If the mixed volume algorithm is
coupled to a path-following algorithm, this technique 
reduces the number of paths
to track.

\paragraph{\bf Hash function.}
In Section~\ref{sec:algorithm}, we associated a unique real number
to each possible lower \revc{face}. 
To a lower \revc{face} $L_{\xi}$ with active constraints labeled by integers
$l_1< \dots< l_{s+d}$
we associated the hash value
\[
\sigma (L_{\xi}) = \sum_{j=1}^{s+d} l_j H_j \mod 1 
\]  
with $H_j$ independently distributed random numbers in $[0,1]$. In the
program, those are replaced by pseudo-random floating point numbers. 
Since pseudo-random 
numbers are a good proxy
for irrational numbers, we expect the values of $\sigma$ to be well spread
from each other \cite{KnuthV3}*{Th.S Sec. 6.4}. 
The reason for choosing the $l_j$ in increasing order is to ensure that
the floating point value associated to a \revc{face} $L_{\xi}$ is always the
same.

\paragraph{\bf Parallelization}
If $N$ processors are available, the $k$-th {\em node} or processor ($0 \le k < N$) is in charge of lower \revc{faces}
$L_{\xi}$ for $k \le N \sigma(L_{\xi}) < k$. The sets $\Vexplore$ and $\Vknown$
are distributed between the processors: each processor stores the 
lower \revc{faces} in its range as a balanced tree. 

Below is a crude version of the parallel algorithm, 
running on node $k$ out of $N$. The new parallel
operations are explained afterwards.

\centerline{\boxed{
\begin{minipage}{\textwidth - 2em}
\begin{trivlist}
\item {\bf Algorithm} {\sc AllMixedCells} (Parallel)\label{allmixedcells2}
\item {\bf If} $k \le N \sigma(L_{F_0}) < k+1$
\subitem {\bf then} $\Vexplore \leftarrow \{ L_{F_0} \}$.
\subitem {\bf else} $\Vexplore \leftarrow \emptyset$.
\item $\Vknown \leftarrow \emptyset$
\item $\mathrm{work\_to\_do} \leftarrow \mathbf{True}$.
\item {\bf While } $\mathrm{work\_to\_do}$
\subitem $L_{\xi} \leftarrow \emptyset$.
\subitem {\bf While }$\Vexplore \ne \emptyset$ {\bf and} $L_{\xi}=\emptyset$,
\subsubitem Remove some $L_{\xi}$ from $\Vexplore$.
\subsubitem {\bf If } {\sc Visited}$(\Vknown, L_{\xi})$ {\bf then} $L_{\xi}=\emptyset$.
\subitem {\bf If }$L_{\xi} \ne \emptyset$,
\subsubitem {\bf If} $L_{\xi}$ is a mixed cell, {\bf then} output $L_{\xi}$.
\subsubitem {\bf For} each neighbor $L_{\xi'}$ of $L_{\xi}$ in $\Gamma$, 
      {\bf send} $L_{\xi'}$ to processor $\lfloor \sigma(L_{\xi'}) / N \rfloor$.
\subsubitem Discard $L_{\xi}$.
\subitem {\bf Wait} until all sent messages are available to the recipient node.
\subitem {\bf Receive} all $L_{\xi}$'s sent to node $k$ and insert them in $\Vexplore$.
\subitem $\mathrm {work\_to\_do} \leftarrow (\Vexplore \not = \emptyset)$
\subitem {\bf Reduce} ( $\mathrm{work\_to\_do}$, {\bf or}, $0 \le k < N$) 
\end{trivlist}
\end{minipage}
}}
\medskip
\par
Parallel machines communicate by sending messages between processors. Operations
{\bf send} and {\bf receive} refer to a message from a given processor, sent to 
a specified processor. Each node can check whether 
a message went through, that is
whether it is available to the recipient. 
It can check for incoming messages.
\par
The {\bf reduce}
operation (modeled on {\sc MPI\_Allreduce}) takes three arguments,
{\em variable}, {\em operation} and {\em range}. 
It allows to efficiently compute an aggregated value
out of a variable at each node in the range, for the given operation. 
In the example above, the variables work 
at each processor are `or-ed', and the result is propagated to all the nodes in the range.

\paragraph{\bf The choice of the function $\mm{i}{d}$}
The choice of the $\mm{i}{d}$ makes a difference. I opted to reorder the $A_i$'s
by increasing dimension, then increasing volume, then increasing number of points.
In the unmixed case, $\mm{i}{d} = 2$ for $i \le d$ and $\mm{i}{d}=1$ when $i>d$.

\paragraph{\bf Numerical stability}
Numerical stability is an issue. Instead of computing rank-1 updates, I opted for
producing the matrix $B$ independently for each lower \revc{face}. Then I stipulated a value of $\epsilon = 10^5 \epsilon_M$
where $\epsilon_M$ is the `machine epsilon' for double precision. The value of
$B$ is always assumed correct, in the sense that
\[
\|B_{\mathrm{computed}} - B_{\mathrm{true}} \|_{\infty}  \le \epsilon \|B_{\mathrm{computed}} \|_{\infty}
.
\]

The bounds for the condition number provided by the Lapack library 
are not always correct. Therefore, I estimated the condition
number of $\Aact$ by 
$\|\Aact\|_{\infty}$ $\|B_{\mathrm{computed}}\|_{\infty}$. 
This was used to bound the forward error 
$\|B_{\mathrm{computed}} - B_{\mathrm{true}} \|_{\infty}$.

\paragraph{\bf Recovery step:} 
When necessary, the precision of the matrix $B$ can be improved by Newton
iteration, where the residual $\Aact B - I$ is computed using long double or quadruple precision using Newton iteration \cite{Demmel}*{Sec.2.5}. This is a very rare event.

\paragraph{\bf Nonstandard numbers:}
Numerators and denominators for each of the $t_l(i,a)$ can be computed with 
absolute error no more than $2 (1 + \max d_i) (n+s) \epsilon$. If the computed absolute
value of the numerator (resp. denominator) of $t_l(i,a)$ is below that bound, 
it is assumed to be zero. Similarly, if $|t_l(i,a) - t_l(i', a')|< \epsilon \| [-e_i\, a] B \mathrm e_j \|^{-1}$,
those quantities are deemed equal.

\paragraph{\bf Random walk:} I experimented with a random walk strategy to find all the connected components of $\revc{G}_{n}$ in $\Gamma$\rev{, instead of the
deterministic walk through all of $\Gamma$}. Each connected
component found was explored by the {\em AllMixedCells} algorithm.
Then other random paths were explored until an heuristic stopping 
criterion was met. The results are discussed on section~\ref{sec:numerical}.

\section{Numerical results}
\label{sec:numerical}
\begin{table}
\centerline{\resizebox{\textwidth}{!}{
\begin{tabular}{||l|rrrl|ll|l||}
\hline \hline
Example       &$ n    $&$  \sum E_i $&    Visited \revc{faces} &$   $T$        $&  AVG $RT$&      Std dev    &   $RT/T$\\ 
\hline
Cyclic13      &$    13$&$      133$&$     3.07\times 10^{06}$&$  7.51\times 10^{10}$&$ 2.06\times 10^{02}$&$      7.73\times 10^{00}$&$   2.74\times 10^{-09}$\\       
Cyclic14      &$    14$&$      157$&$     1.10\times 10^{07}$&$  3.66\times 10^{11}$&$ 8.50\times 10^{02}$&$      2.79\times 10^{01}$&$   2.33\times 10^{-09}$\\       
Cyclic15      &$    15$&$      183$&$     4.40\times 10^{07}$&$  1.96\times 10^{12}$&$ 4.07\times 10^{03}$&$      1.22\times 10^{02}$&$   2.08\times 10^{-09}$\\       
\hline
Noon18        &$    18$&$      324$&$     8.10\times 10^{06}$&$  8.92\times 10^{11}$&$ 1.23\times 10^{03}$&$      1.42\times 10^{02}$&$   1.38\times 10^{-09}$\\       
Noon19        &$    19$&$      361$&$     1.64\times 10^{07}$&$  2.24\times 10^{12}$&$ 2.87\times 10^{03}$&$      4.08\times 10^{02}$&$   1.28\times 10^{-09}$\\       
Noon20        &$    20$&$      400$&$     3.37\times 10^{07}$&$  5.62\times 10^{12}$&$ 6.46\times 10^{03}$&$      6.53\times 10^{02}$&$   1.15\times 10^{-09}$\\       
\hline
Chandra18     &$    18$&$      307$&$     3.32\times 10^{06}$&$  3.45\times 10^{11}$&$ 5.18\times 10^{02}$&$      3.63\times 10^{01}$&$   1.50\times 10^{-09}$\\       
Chandra19     &$    19$&$      343$&$     7.17\times 10^{06}$&$  9.27\times 10^{11}$&$ 1.27\times 10^{03}$&$      5.44\times 10^{01}$&$   1.37\times 10^{-09}$\\       
Chandra20     &$    20$&$      381$&$     1.57\times 10^{07}$&$  2.50\times 10^{12}$&$ 3.08\times 10^{03}$&$      1.96\times 10^{02}$&$   1.23\times 10^{-09}$\\       
Chandra21     &$    21$&$      421$&$     3.46\times 10^{07}$&$  6.67\times 10^{12}$&$ 7.58\times 10^{03}$&$      4.59\times 10^{02}$&$   1.14\times 10^{-09}$\\       
\hline
Katsura15     &$    16$&$      200$&$     5.03\times 10^{06}$&$  2.73\times 10^{11}$&$ 5.57\times 10^{02}$&$      6.15\times 10^{01}$&$   2.04\times 10^{-09}$\\       
Katsura16     &$    17$&$      225$&$     1.41\times 10^{07}$&$  9.72\times 10^{11}$&$ 1.88\times 10^{03}$&$      2.69\times 10^{02}$&$   1.93\times 10^{-09}$\\       
Katsura17     &$    18$&$      252$&$     3.55\times 10^{07}$&$  3.05\times 10^{12}$&$ 5.31\times 10^{03}$&$      5.57\times 10^{02}$&$   1.74\times 10^{-09}$\\       
Katsura18     &$    19$&$      280$&$     8.30\times 10^{07}$&$  8.81\times 10^{12}$&$ 1.42\times 10^{04}$&$      9.02\times 10^{02}$&$   1.61\times 10^{-09}$\\       
\hline
Gaukwa7       &$    14$&$       98$&$     9.78\times 10^{05}$&$  2.01\times 10^{10}$&$ 7.00\times 10^{01}$&$      4.83\times 10^{00}$&$   3.48\times 10^{-09}$\\       
Gaukwa8       &$    16$&$      128$&$     1.07\times 10^{07}$&$  3.73\times 10^{11}$&$ 1.02\times 10^{03}$&$      9.30\times 10^{01}$&$   2.73\times 10^{-09}$\\       
\hline
Vortex5       &$    10$&$      120$&$     1.37\times 10^{05}$&$  1.83\times 10^{09}$&$ 7.62\times 10^{00}$&$      6.47\times 10^{-01}$&$   4.16\times 10^{-09}$\\       
Vortex6       &$    15$&$      240$&$     7.53\times 10^{07}$&$  4.39\times 10^{12}$&$ 7.25\times 10^{03}$&$      9.17\times 10^{02}$&$   1.65\times 10^{-09}$\\       
\hline
N-body5       &$    10$&$      130$&$     2.94\times 10^{05}$&$  4.30\times 10^{09}$&$ 1.42\times 10^{01}$&$      1.78\times 10^{00}$&$   3.31\times 10^{-09}$\\       
\hline
Gridanti3     &$    18$&$       81$&$     7.10\times 10^{04}$&$  1.93\times 10^{09}$&$ 9.26\times 10^{00}$&$      4.28\times 10^{-01}$&$   4.80\times 10^{-09}$\\       
Gridanti4     &$    32$&$      144$&$     3.08\times 10^{07}$&$  4.62\times 10^{12}$&$ 1.14\times 10^{04}$&$      5.98\times 10^{02}$&$   2.47\times 10^{-09}$\\       
\hline
Sonic8        &$    28$&$      407$&$     6.78\times 10^{06}$&$  2.21\times 10^{12}$&$ 2.49\times 10^{03}$&$      4.37\times 10^{02}$&$   1.13\times 10^{-09}$\\       
Sonic9        &$    30$&$      452$&$     1.38\times 10^{07}$&$  5.70\times 10^{12}$&$ 5.89\times 10^{03}$&$      9.25\times 10^{02}$&$   1.03\times 10^{-09}$\\       
Sonic10       &$    32$&$      497$&$     2.34\times 10^{07}$&$  1.21\times 10^{13}$&$ 1.18\times 10^{04}$&$      1.56\times 10^{03}$&$   9.73\times 10^{-10}$\\       
\hline
Graphmodel6   &$    21$&$       48$&$     3.37\times 10^{04}$&$  7.29\times 10^{08}$&$ 6.50\times 10^{00}$&$      6.65\times 10^{-01}$&$   8.92\times 10^{-09}$\\       
Graphmodel7   &$    28$&$       63$&$     4.14\times 10^{05}$&$  2.08\times 10^{10}$&$ 1.11\times 10^{02}$&$      7.52\times 10^{00}$&$   5.33\times 10^{-09}$\\       
Graphmodel8   &$    36$&$       80$&$     7.43\times 10^{06}$&$  7.80\times 10^{11}$&$ 3.85\times 10^{03}$&$      3.38\times 10^{02}$&$   4.94\times 10^{-09}$\\       
\hline
Eco20         &$    20$&$      209$&$     1.11\times 10^{07}$&$  9.61\times 10^{11}$&$ 1.93\times 10^{03}$&$      5.25\times 10^{02}$&$   2.00\times 10^{-09}$\\       
Eco21         &$    21$&$      230$&$     2.36\times 10^{07}$&$  2.48\times 10^{12}$&$ 4.62\times 10^{03}$&$      1.30\times 10^{03}$&$   1.86\times 10^{-09}$\\       
\hline
Reimer13      &$    13$&$      169$&$     6.99\times 10^{03}$&$  2.10\times 10^{08}$&$ 1.54\times 10^{00}$&$      1.22\times 10^{-01}$&$   7.33\times 10^{-09}$\\       
Reimer14      &$    14$&$      196$&$     1.15\times 10^{04}$&$  4.62\times 10^{08}$&$ 2.18\times 10^{00}$&$      2.51\times 10^{-01}$&$   4.72\times 10^{-09}$\\       
Reimer15      &$    15$&$      225$&$     1.92\times 10^{04}$&$  1.02\times 10^{09}$&$ 3.35\times 10^{00}$&$      2.20\times 10^{-01}$&$   3.29\times 10^{-09}$\\       
\hline
VortexAC4     &$    12$&$       84$&$     1.17\times 10^{04}$&$  1.50\times 10^{08}$&$ 1.55\times 10^{00}$&$      6.69\times 10^{-02}$&$   1.03\times 10^{-08}$\\       
VortexAC5     &$    20$&$      200$&$     3.19\times 10^{06}$&$  2.64\times 10^{11}$&$ 4.84\times 10^{02}$&$      5.09\times 10^{01}$&$   1.83\times 10^{-09}$\\       
\hline
\hline
\end{tabular}}}
\caption{Predicted value of $T$ and measured running time for selected examples \revc{(8 cores)}.\label{table:running}}
\end{table}

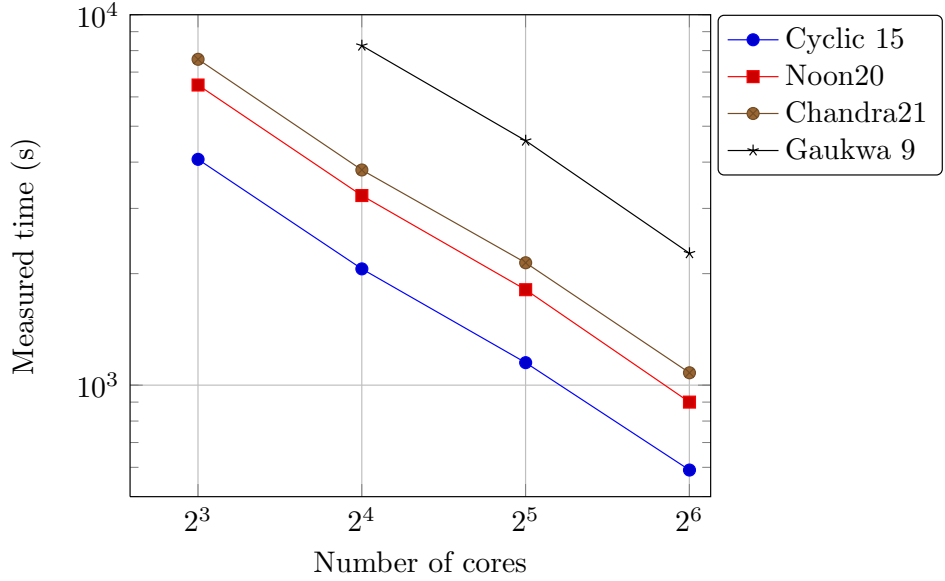
\begin{figure}
\centerline{\resizebox{\textwidth}{!}{
\begin{tikzpicture}
\begin{loglogaxis}[
	xlabel={Number of cores},
	ylabel={Measured time (s)},
        xmin=6, xmax=70,
        log basis x=2, 
        ymin=500,ymax=10000,
        grid=major,
        legend style={cells={anchor=west}, rounded corners=2pt, at={(1.4,1.0)}}
        ]
\addplot coordinates{(8,4.07E+3) (16,2.06E+3) (32,1.15E+3) (64,5.9E+2)};
\addplot coordinates{(8,6.46E+3) (16,3.25E+3) (32,1.81E+3) (64,9E+2)};
\addplot coordinates{(8,7.58E+3) (16,3.81E+3) (32,2.14E+3) (64,1.08E+3)};
\addplot coordinates{(16,8.25E+3) (32,4.57E+3) (64,2.27E+3)};
\legend{Cyclic 15, Noon20, Chandra21, Gaukwa 9}
\end{loglogaxis}
\end{tikzpicture}}}
\caption{Measured running time (using 8 cores) against
the invariant $T$ from equation \eqref{main-time1} for several benchmark examples.
Data from table~\ref{table:running} page~\pageref{table:running}.
\label{plot:speedup}}
\end{figure}

\begin{table}
\centerline{\resizebox{\textwidth}{!}{
\begin{tabular}{lrrrrr}
\hline \hline
Number of cores: & $8$ & $16$ & $32$ & $64$ & Efficiency exp.\\ 
\hline
Cyclic15    & $    4.07\times 10^{03}$ & $    2.06\times 10^{03}$ & $    1.15\times 10^{03}$ & $    5.90\times 10^{02}$ & $    9.20\times 10^{-01}$\\        
Noon20      & $    6.46\times 10^{03}$ & $    3.25\times 10^{03}$ & $    1.81\times 10^{03}$ & $    9.00\times 10^{02}$ & $    9.37\times 10^{-01}$\\        
Chandra21   & $    7.58\times 10^{03}$ & $    3.81\times 10^{03}$ & $    2.14\times 10^{03}$ & $    1.08\times 10^{03}$ & $    9.28\times 10^{-01}$\\        
Gaukwa9     &  Unavail. & $    8.25\times 10^{03}$ & $    4.57\times 10^{03}$ & $    2.27\times 10^{03}$ & $    9.30\times 10^{-01}$\\
\hline
\hline
\end{tabular}}}
\caption{Parallel timing. The efficiency exponent is the 
quadratic best fit for the linear coefficient of 
$(\log(N_i),-\log (T_i))$ where $T_i$ is the measured running
time with $N_i$ cores.\label{table:scaling}}
\end{table}

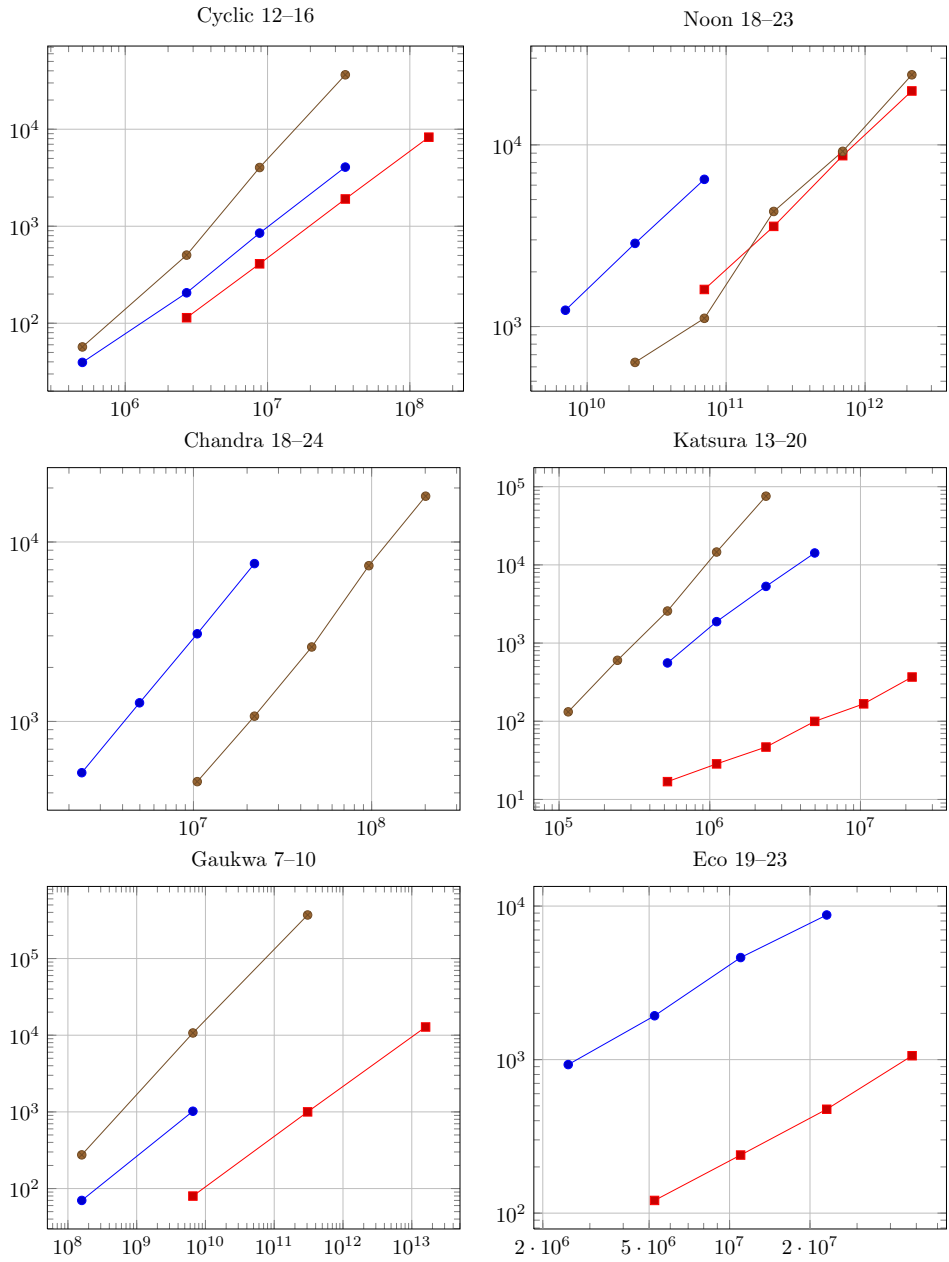
\begin{figure}
\centerline{\resizebox{\textwidth}{!}{
\begin{tikzpicture}
\begin{loglogaxis}[
	title={Cyclic 12--\revb{16}},
    grid=major,
    ]
\addplot coordinates{(    5.00E+05,    3.95E+01) (    2.70E+06,    2.06E+02) (    8.80E+06,    8.50E+02) (    3.52E+07,    4.07E+03) };
\addplot coordinates{(    2.70E+06,    1.14E+02)  (    8.80E+06,    4.10E+02)   (    3.52E+07,    1.91E+03)  (1.36E+08,8.28E+03)};
\addplot coordinates{  (    5.00E+05,    5.70E+01) (    2.70E+06,    5.04E+02) (    8.80E+06,    4.03E+03) (    3.52E+07,    3.64E+04)};
\end{loglogaxis}
\end{tikzpicture}\hskip 10pt
\begin{tikzpicture}
\begin{loglogaxis}[
	title={Noon 18--23},
    grid=major,
    ]
\addplot coordinates{(    6.97E+09,    1.23E+03) (    2.21E+10,    2.87E+03) (    6.97E+10,    6.46E+03)  };
\addplot coordinates{(6.97E+10, 1.60E+03) (    2.20E+11,    3.56E+03)  (    6.90E+11,    8.72E+03)  (	2.17E+12	,	1.98E+04	) };
\addplot coordinates{(    2.21E+10,    6.35E+02)(    6.97E+10,    1.11E+03)(    2.20E+11,    4.30E+03)(    6.90E+11,    9.21E+03) (    2.17E+12,    2.43E+04) };
\end{loglogaxis}
\end{tikzpicture}}}
\centerline{\resizebox{\textwidth}{!}{
\begin{tikzpicture}
\begin{loglogaxis}[
	title={Chandra 18--24},
    grid=major,
    ]
\addplot coordinates{(    2.36E+06,    5.18E+02) (    4.98E+06,    1.27E+03) (    1.05E+07,    3.08E+03) (    2.20E+07,    7.58E+03) };
\addplot coordinates{};
\addplot coordinates{(    1.05E+07,    4.62E+02) (    2.20E+07,    1.07E+03)(    4.61E+07,    2.60E+03)(    9.65E+07,    7.38E+03)(    2.01E+08,    1.80E+04)};
\end{loglogaxis}
\end{tikzpicture}
\hskip 10pt
\begin{tikzpicture}
\begin{loglogaxis}[
	title={Katsura 13--\revb{20}},
    grid=major,
    ]
\addplot coordinates{(    5.24E+05,    5.57E+02) (    1.11E+06,    1.88E+03) (    2.36E+06,    5.31E+03) (    4.97E+06,    1.42E+04) };
\addplot coordinates{ (5.24E+05,1.69E+01) (1.11E+06, 2.85E+01) (    2.36E+06,    4.68E+01) (    4.97E+06,    9.97E+01)  (    1.05E+07,    1.67E+02) (2.20E+07,3.68E+02) };
\addplot coordinates{(    1.15E+05,    1.32E+02)(    2.44E+05,    6.02E+02)
(    5.24E+05,    2.57E+03)(    1.11E+06,    1.46E+04)(    2.36E+06,    7.56E+04)};
\end{loglogaxis}
\end{tikzpicture}}}
\centerline{\resizebox{\textwidth}{!}{
\begin{tikzpicture}
\begin{loglogaxis}[
	title={Gaukwa 7--\revb{10}},
    grid=major,
    ]
\addplot coordinates{(    1.59E+08,    7.00E+01) (    6.57E+09,    1.02E+03)           };
\label{plot:amc}
\addplot coordinates{ (    6.57E+09,    7.98E+01) (    3.06E+11,    1.00E+03) (1.59E+13, 1.28E+04) };
\label{plot:rp}
\addplot coordinates{ (    1.59E+08,    2.75E+02)(    6.57E+09,    1.07E+04) (    3.06E+11,    3.70E+05)};
\label{plot:ll}
\end{loglogaxis}
\end{tikzpicture}
\hskip 10pt
\begin{tikzpicture}
\begin{loglogaxis}[
	title={Eco 19--\revb{23}},
    xtickten={6,7},
    extra x ticks={2E+6, 5E+6, 2E+7},
    grid=major,
    ]
\addplot coordinates{(    2.49E+06,    9.28E+02) (    5.24E+06,    1.93E+03) (    1.10E+07,    4.62E+03) (    2.31E+07,    8.75E+03) };
\addplot coordinates{ (    5.24E+06,    1.21E+02) (    1.10E+07,   2.39E+02)  (    2.31E+07,    4.75E+02) (4.82E+07, 1.06E+03) };
\addplot coordinates{};
\end{loglogaxis}
\end{tikzpicture}}}
\caption{Measured running time of {\sc AllMixedCells} (\ref{plot:amc}), {\sc AllMixedCells} with random path (\ref{plot:rp}) and running time published by
\ocite{Lee-Li} (\ref{plot:ll}) against the output size for selected examples.
All times in seconds.
Data from table~\ref{table:sensitivity} page~\pageref{table:sensitivity}.
\label{plot:sensitivity}}
\end{figure}

\begin{table}
\centerline{\resizebox{\textwidth}{!}{
\begin{tabular}{||l|rrrl|lll||}
\hline \hline
Example    &  $n$ &   Mixed volume &   det  &   Output size& AllMixedCells&Random path&Lee and Li \\         
& & & & & \revc{(8 cores)} & \revc{(8 cores)} & \revc{(1 core)} \\
\hline 
Cyclic12	&$	12	$&$	500352	$&$	12	$&$	5.00\times 10^{05	}$&$	3.95\times 10^{01	}$&$1.91 \times 10^{01		}$&$	5.70\times 10^{01	}$\\
Cyclic13	&$	13	$&$	2704156	$&$	13	$&$	2.70\times 10^{06	}$&$	2.06\times 10^{02	}$&$	1.14\times 10^{02	}$&$	5.04\times 10^{02	}$\\
Cyclic14	&$	14	$&$	8795976	$&$	14	$&$	8.80\times 10^{06	}$&$	8.50\times 10^{02	}$&$	4.10\times 10^{02	}$&$	4.03\times 10^{03	}$\\
Cyclic15	&$	15	$&$	35243520	$&$	15	$&$	3.52\times 10^{07	}$&$	4.07\times 10^{03	}$&$	1.91\times 10^{03	}$&$	3.64\times 10^{04	}$\\
Cyclic16	&$	16	$&$	135555072	$&$	16	$&$	1.36\times 10^{08	}$&$	$&$	8.28\times 10^{03	}$&$	$\\
\hline 
Noon18	&$	18	$&$	387420453	$&$	1	$&$	6.97\times 10^{09	}$&$	1.23\times 10^{03	}$&$		$&$		$\\
Noon19	&$	19	$&$	1162261429	$&$	1	$&$	2.21\times 10^{10	}$&$	2.87\times 10^{03	}$&$		$&$	6.35\times 10^{02	}$\\
Noon20	&$	20	$&$	3486784361	$&$	1	$&$	6.97\times 10^{10	}$&$	6.46\times 10^{03	}$&$		$&$	1.11\times 10^{03	}$\\
Noon21	&$	21	$&$	10460353161	$&$	1	$&$	2.20\times 10^{11	}$&$		$&$	3.56\times 10^{03	}$&$	4.30\times 10^{03	}$\\
Noon22	&$	22	$&$	31381059565	$&$	1	$&$	6.90\times 10^{11	}$&$		$&$	8.72\times 10^{03	}$&$	9.21\times 10^{03	}$\\
Noon23	&$	23	$&$	94143178781	$&$	1	$&$	2.17\times 10^{12	}$&$		$&$1.98\times{10^4}		$&$	2.43\times 10^{04	}$\\
\hline 
Chandra18	&$	18	$&$	131072	$&$	1	$&$	2.36\times 10^{06	}$&$	5.18\times 10^{02	}$&$		$&$		$\\
Chandra19	&$	19	$&$	262144	$&$	1	$&$	4.98\times 10^{06	}$&$	1.27\times 10^{03	}$&$		$&$		$\\
Chandra20	&$	20	$&$	524288	$&$	1	$&$	1.05\times 10^{07	}$&$	3.08\times 10^{03	}$&$		$&$	4.62\times 10^{02	}$\\
Chandra21	&$	21	$&$	1048576	$&$	1	$&$	2.20\times 10^{07	}$&$	7.58\times 10^{03	}$&$		$&$	1.07\times 10^{03	}$\\
Chandra22	&$	22	$&$	2097152	$&$	1	$&$	4.61\times 10^{07	}$&$		$&$		$&$	2.60\times 10^{03	}$\\
Chandra23	&$	23	$&$	4194304	$&$	1	$&$	9.65\times 10^{07	}$&$		$&$		$&$	7.38\times 10^{03	}$\\
Chandra24	&$	24	$&$	8388608	$&$	1	$&$	2.01\times 10^{08	}$&$		$&$		$&$	1.80\times 10^{04	}$\\
\hline 
Katsura13	&$	14	$&$	8190	$&$	1	$&$	1.15\times 10^{05	}$&$		$&$		$&$	1.32\times 10^{02}$\\
Kastura14	&$	15	$&$	16254	$&$	1	$&$	2.44\times 10^{05	}$&$		$&$		$&$	6.02\times 10^{02}$\\
Katsura15	&$	16	$&$	32730	$&$	1	$&$	5.24\times 10^{05	}$&$	5.57\times 10^{02	}$&$	1.69\times 10^{01}	$&$	2.57\times 10^{03	}$\\
Katsura16	&$	17	$&$	65280	$&$	1	$&$	1.11\times 10^{06	}$&$	1.88\times 10^{03	}$&$	2.85\times 10^{01}	$&$	1.46\times 10^{04	}$\\
Katsura17	&$	18	$&$	131070	$&$	1	$&$	2.36\times 10^{06	}$&$	5.31\times 10^{03	}$&$	4.68\times 10^{01}$&$	7.56\times 10^{04	}$\\
Katsura18	&$	19	$&$	261576	$&$	1	$&$	4.97\times 10^{06	}$&$	1.42\times 10^{04	}$&$	9.97\times 10^{01}$&$		$\\
Katsura19	&$	20	$&$	524286	$&$	1	$&$	1.05\times 10^{07	}$&$		$&$	1.67\times 10^{02	}$&$		$\\
Kastura20	&$	21	$&$	1047540	$&$	1	$&$	2.20\times 10^{07	}$&$		$&$	3.68\times 10^{02	}$&$		$\\
\hline 
Gaukwa7	&$	14	$&$	11390625	$&$	1	$&$	1.59\times 10^{08	}$&$	7.00\times 10^{01	}$&$	9.14\times 10^{00	}$&$	2.75\times 10^{02	}$\\
Gaukwa8	&$	16	$&$	410338673	$&$	1	$&$	6.57\times 10^{09	}$&$	1.02\times 10^{03	}$&$	7.98\times 10^{01	}$&$	1.07\times 10^{04	}$\\
Gaukwa9	&$	18	$&$	16983563041	$&$	1	$&$	3.06\times 10^{11	}$&$		$&$	1.00\times 10^{03	}$&$	3.70\times 10^{05	}$\\
\revb{Gaukwa10}	&$	20	$&$	794280046581	$&$	1	$&$	1.59 \times10^{13	}$&$		$&$	1.28\times 10^{04}$&$		$\\
\hline
Eco19	&$	19	$&$	131072	$&$	1	$&$	2.49\times 10^{06	}$&$	9.28\times 10^{02	}$&$		$&$		$\\
Eco20	&$	20	$&$	262144	$&$	1	$&$	5.24\times 10^{06	}$&$	1.93\times 10^{03	}$&$	1.21\times 10^{02	}$&$		$\\
Eco21	&$	21	$&$	524288	$&$	1	$&$	1.10\times 10^{07	}$&$	4.62\times 10^{03	}$&$	2.39\times 10^{02	}$&$		$\\
Eco22	&$	22	$&$	1048576	$&$	1	$&$	2.31\times 10^{07	}$&$	8.75\times 10^{03	}$&$	4.75\times 10^{02	}$&$		$\\
Eco23	&$	23	$&$	2097152	$&$	1	$&$	4.82\times 10^{07	}$&$	$&$	1.06\times 10^{03	}$&$		$\\
\hline 
\hline
\end{tabular}}}
\caption{Measured output-sensitivity of the full {\sc AllMixedCells} algorithm, of {\sc AllMixedCells} with the random
path heuristic and results reported by \ocite{Lee-Li}.\label{table:sensitivity}}
\end{table}

\subsection{Choice of the examples.}

I selected a few families of benchmark systems for which a general
formula is available or is easy to figure. In particular, the benchmark
families \rev{tested} by \ocite{Lee-Li} were all included in the benchmark:
{\em Cyclic}, {\em Noon}, {\em Chandra}, 
{\em Katsura}, {\em Gaukwa}, {\em Vortex}, {\em N-body},
{\em Gridanti} and {\em Sonic}. \rev{Precise references for most of them
were given by \ocite{Verschelde-795}.}

The system {\em VortexAC} from ~\cite{Chen-Lee-Li} was also included for
the sake of comparison. Due to hardware limitations, I was unable to
compute the mixed volume of VortexAC-6.

Jan Verschelde maintains a list of polynomial systems in 
\url{http://homepages.math.uic.edu/~jan/demo.html}. From his list
I selected the {\em Eco} and the {\em Reimer} families. 

According to \ocite{Morgan}*{p.148(7.3)}, the system  
{\em Eco-$n$} arised from economic modeling.
It is defined by 
\begin{align*}
-c_i + x_i x_n + \sum_{j=1}^{n-i-1} x_jx_{i+j}x_n &= 0 &  &\text{for }i=1, \dots, n-1,
\\
x_1 + \cdots + x_{n-1} + 1 &= 0 & &
\end{align*}

The Reimer-$n$ family was defined in the Posso suite, still available
at {http://www-sop.inria.fr/saga/POL/}. It is given by:
\begin{align*}
-1/2 + \sum_{j=1}^{n}(-1)^{j+1}x_{j}^{i+1} &= 0 &&\text{for } 1 \le i \le n.
\end{align*}

The family {\em Graphmodel} comes from Gaussian graphical models in
statistics \cite{Uhler}.
Let $G$ be the cyclic graph with vertices
$\{1, \dots, n\}$ and edges $\mathcal E =
\{ \{1,1\}, \{2,2\} \dots, \{n,n\} 
\cup \{ \{1,2\}, \{2,3\}, \cdots, \{n,1\} \}$. Let
$X$ and $Y$ be $n \times n$ symmetric matrices, constructed
as follows. Assume that $i \le j$:
If $\{i,j\} \in \mathcal E$, then $X_{ij}$ is a variable and
$Y_{ij}$ is a random complex number. If $(i,j) \not \in \mathcal E$,
then $Y_{ij}$ is a variable and $X_{ij}=0$.

The system {\em Graphmodel-n} is given by the \rev{upper triangular part 
of the} matrix equation
$X Y = I$. The actual number of roots of \rev{the overdetermined system
$X Y = I$} is known as
the {\em Maximum Likelihood degree} of the graphical model $G$. 
The same construction may be carried out for any graph $G$.

\subsection{Hardware}
Computations were performed at NACAD ({\em Núcleo Avançado de 
Computação de Alto Desempenho}) in the {\em Universidade Federal do Rio de Janeiro}. 
The machine used was a SGI Altix ICE 8400 running Intel MKL (includes Lapack) and MVAPICH2 (MPI implementation). I used up to 8 nodes, with 8 cores per node. \revc{The CPUs are either Six Core Intel Xeon X5650 (Westmere) running at 2.67 GHz
or Quad Core Intel Xeon X5355 (Clovertown) running at 2.66 GHz.}

\subsection{Claim 1: adequacy of the computation model}
Each of the benchmark examples was tested for 10 pseudo-random liftings. 
The results are displayed in table~\ref{table:running}. 
The number of visited \revc{faces} is an average.
The column $T$ displays the bound in \eqref{main-time1}. 
The running time is the average wall time for all steps of the program, 
from reading the input to writing the mixed cells to the output file.
For systems large enough, the running time was found to be of the
order of $10^{-9} T$ seconds, with a standard deviation under 15\%.

\subsection{Claim 2: scaling} Four benchmark examples with $T>10^{12}$
were selected for the scaling test (Table~\ref{table:scaling}). 

In order to evaluate the efficiency of parallelization, I used least squares
to compute the linear coefficient of the best affine approximation for data
$(\log N_i, -\log T_{i})$ where $N_i$ is the number of cores and $T_i$
the average running time. I obtained a running time of
\[
O(N^{0.93})
\]
while perfect, linear parallelization would yield $O(N)$. 

\subsection{Claim 3: comparison with other available software}

The best published timings for finding mixed cells are those
in~\cite{Lee-Li} and ~\cite{Chen-Lee-Li}. Since experiments were
performed in different machines, the absolute timings may not be
comparable. However, the time ratio from a benchmark example to the
next example in the same family is an invariant.

To make sense from this invariant, I plotted the running time against
the output size, in a log-log scale (Fig.~\ref{plot:sensitivity}).

The slopes of the lines show how the running time increases with respect to the output size. A slope close
to one or smaller implies that computing the mixed cells will not
be a bottleneck for the overall polynomial solving by homotopy.

Fig. ~\ref{plot:sensitivity} shows that there is not a best algorithm
for all cases and that all the three tested algorithms are competitive for
some of the benchmark families. It is possible to see an asymptotic
gain in running time for the families {\em Cyclic}, {\em Katsura}
and {\em Gaukwa} when using {\sc AllMixedCellsFull}. 

Comparison with the results by \cite{Chen-Lee-Li} can only be done
in terms of absolute running time, adjusting their results in the
shared memory model to 8 cores. Their program had a 
similar running time for the {\em Cyclic-15} and {\em Eco-20} 
examples and was faster than {\sc AllMixedCellsFull} 
for {\em Sonic-8} and {\em Katsura-15}.
However the number of cores using shared memory is limited so those
results are not necessarily scalable.

\section{Conclusions}\label{sec:conclusions}
We introduced a new algorithm to compute mixed cells and mixed volumes. Its running time was bounded in terms of
quermassintegrals associated to the supporting polytopes of the equations. This
is the first non-combinatorial bound for mixed volume computation.
\medskip
\par
The implementation of the algorithm is competitive with available
software. Its main drawbacks are memory usage and some numerical stability
problems for very large polynomial systems.
\par
Memory usage problems disappear when using a sufficient number of
processors, since most memory storage is
local and distributed.
\par
Numerical instability arises when two \revc{faces} are nearly parallel, or when
the matrix of active constraints is nearly degenerate.
This problem was
solved through rigorous error bounds and judicious use of 
quadruple precision arithmetic. 
\revc{
Extra precision may be required if the number of faces to visit
becomes substantially larger than in the tested examples (table~\ref{table:running}).
At this time precision is not an issue, so this is left for future implementations.
}
\par
The random walk method for accelerating the algorithm is a promising
strategy. It should be coupled with a fast mixed volume estimator
(unavailable at this time) to ensure correctness of the results. Moreover,
 random graph search algorithms are a research subject by
itself.
\par
\rev{While the motivation of this paper was to provide good starting systems for
homotopy, the numerical implementation of polyhedral homotopy continuation may require
adequate mathematical machinery beyond projective spaces and unitary group 
action. See for instance \cite{Malajovich-Rojas,
Malajovich-Fewspaces} on conditioning and root counting on toric varieties.
Homotopy algorithms on toric varieties will be 
the subject of a future paper.}

\revd{
\begin{note*} While this paper was under review, 
\ocite{Jensen} proposed a symbolic algorithm for tropical
homotopy continuation, using similar but subtly different
ideas. Preliminary experiments suggest 
a running time comparable to the
random path method (table~\ref{table:sensitivity}), yet it
is deterministic.
\end{note*}
}

\section*{Glossary of notations}
\centerline{\resizebox{\textwidth}{!}{
\begin{theglossary}
\input{mixed.gls}
\end{theglossary}
}}

\end{document}

%% file: legendre.pdf_t
\begin{picture}(0,0)%
\includegraphics{legendre.pdf}%
\end{picture}%
\setlength{\unitlength}{4144sp}%
\begingroup\makeatletter\ifx\SetFigFont\undefined%
\gdef\SetFigFont#1#2#3#4#5{%
  \reset@font\fontsize{#1}{#2pt}%
  \fontfamily{#3}\fontseries{#4}\fontshape{#5}%
  \selectfont}%
\fi\endgroup%
\begin{picture}(7694,4319)(429,-3233)
\put(7975,-1785){\makebox(0,0)[lb]{\smash{{\SetFigFont{17}{20.4}{\familydefault}{\mddefault}{\updefault}{\color[rgb]{0,0,0}$\xi$}%
}}}}
\put(5449,-265){\makebox(0,0)[lb]{\smash{{\SetFigFont{17}{20.4}{\familydefault}{\mddefault}{\updefault}{\color[rgb]{0,0,0}$\lambda$}%
}}}}
\put(3809,-2592){\makebox(0,0)[lb]{\smash{{\SetFigFont{17}{20.4}{\familydefault}{\mddefault}{\updefault}{\color[rgb]{0,0,0}$x$}%
}}}}
\put(549, 68){\makebox(0,0)[lb]{\smash{{\SetFigFont{17}{20.4}{\familydefault}{\mddefault}{\updefault}{\color[rgb]{0,0,0}$b$}%
}}}}
\put(3302,-45){\makebox(0,0)[lb]{\smash{{\SetFigFont{17}{20.4}{\familydefault}{\mddefault}{\updefault}{\color[rgb]{0,0,0}$\hat b(x)$}%
}}}}
\put(3389,-1425){\makebox(0,0)[lb]{\smash{{\SetFigFont{17}{20.4}{\familydefault}{\mddefault}{\updefault}{\color[rgb]{0,0,0}$b(a)$}%
}}}}
\put(1801,-2671){\makebox(0,0)[lb]{\smash{{\SetFigFont{17}{20.4}{\familydefault}{\mddefault}{\updefault}{\color[rgb]{0,0,0}$A$}%
}}}}
\put(6211,-871){\makebox(0,0)[lb]{\smash{{\SetFigFont{17}{20.4}{\familydefault}{\mddefault}{\updefault}{\color[rgb]{0,0,0}$\lambda(\xi)$}%
}}}}
\end{picture}%